\documentclass[12pt]{amsart}
\usepackage{latexsym}
\usepackage{amssymb} 
\usepackage{mathrsfs}
\usepackage{amsmath}
\usepackage{latexsym}
\usepackage{delarray}
\usepackage{amssymb,amsmath,amsfonts,amsthm,
mathrsfs}

\usepackage{hyperref}
\renewcommand\eqref[1]{(\ref{#1})} 

\newtheorem{theorem}{Theorem}[section]

\newtheorem{lemma}[theorem]{Lemma}
\newtheorem{proposition}[theorem]{Proposition}
\newtheorem{definition}[theorem]{Definition}
\theoremstyle{definition}
\newtheorem{remark}[theorem]{Remark}

\newcommand{\wt}[1]{\widetilde{#1}}

\newcommand{\Cinf}{\ensuremath{\mathcal{C}^\infty}}
\newcommand{\Cinfc}{\ensuremath{\mathcal{C}^\infty_{\text{c}}}}
\newcommand{\D}{\ensuremath{{\mathcal D}}}
\renewcommand{\S}{\mathscr{S}}
\newcommand{\E}{\ensuremath{{\mathcal E}}}

\newcommand{\mF}{\mathcal{F}}

\newcommand{\mb}[1]{\ensuremath{\mathbb{#1}}}
\newcommand{\N}{\mb{N}}

\newcommand{\R}{\mb{R}}
\newcommand{\C}{\mb{C}}

\newcommand{\lara}[1]{\langle #1 \rangle}

\newfont{\bl}{msbm10 scaled \magstep2}

\newcommand{\beq}{\begin{equation}}
\newcommand{\eeq}{\end{equation}}

\newcommand{\eps}{\varepsilon}

\renewcommand{\Im}{\ensuremath{\mathrm{Im}}}

\newcommand{\esp}{\mathrm{e}}
\newcommand\Rn{{\mathbb R}^n}

\newcommand{\supp}{\mathrm{supp}}

\renewcommand\N{{\mathbb N}_0}

\title[On hyperbolic equations and systems]{
On hyperbolic equations and systems with non-regular time dependent coefficients}

\author[Claudia Garetto]{Claudia Garetto}
\address{
  Claudia Garetto:
  \endgraf
  Department of Mathematical Sciences
  \endgraf
  Loughborough University
  \endgraf
  Loughborough, Leicestershire, LE11 3TU
  \endgraf
  United Kingdom
  \endgraf
  {\it E-mail address} {\rm c.garetto@lboro.ac.uk}
  }

\thanks{The author was supported by the
EPSRC First grant EP/L026422/1.
}
\date{}

\subjclass[2010]{Primary 35L25; 35L40; Secondary 46F05;}
\keywords{Hyperbolic equations, Gevrey spaces, ultradistributions, weak solutions}

\begin{document}

\maketitle

\begin{abstract}
In this paper we study higher order weakly hyperbolic equations with
time dependent non-regular coefficients. The non-regularity here means less than 
H\"older, namely bounded coefficients.
As for second order equations in \cite{GR:14} we prove that such equations admit a
`very weak solution' adapted to the type of solutions that exist for
regular coefficients. The main idea in the construction of a very weak solution is the regularisation
of the coefficients via convolution with a mollifier and a qualitative analysis of the corresponding
family of classical solutions depending on the regularising parameter.
Classical solutions are recovered as limit of
very weak solutions. Finally, by using a reduction to block Sylvester form we conclude that any first order hyperbolic system with non-regular coefficients
is solvable in the very weak sense.
 
\end{abstract}

\section{Introduction}
We want to study equations of the type
\beq
\label{intro_eq}
D^m_t u-\sum_{j=1}^m\sum_{|\nu|= j} a_{\nu,j}(t)D_t^{m-j}D^\nu_x u-\sum_{j=1}^m\sum_{|\nu|<j}b_{\nu,j}(t)D^{m-j}_tD^\nu_x u=f(t,x),\,t\in[0,T], x\in\R^n,
\eeq
under initial conditions
\beq
\label{intro_ic}
D_t^{k}u(0,x)=g_k,\quad k=0,\cdots,m-1.
\eeq
 We assume that the roots of the characteristic polynomial 
 \[
 \tau^m-\sum_{j=1}^m\sum_{|\nu|= j} a_{\nu,j}(t)\xi^\nu\tau^{m-j}=\Pi_{j=1,\dots,m}(\tau-\lambda_j(t,\xi))
 \]
 are real and bounded in $t$ but not necessarily regular, for instance they might be discontinuous in $t$ as generated by discontinuous coefficients  $a_{\nu,m}$.  We assume that the coefficients of the lower order terms are compactly supported distributions with support contained in $[0,T]$, the right-hand side $f$ belongs to $\E'([0,T])\otimes \E'(\Rn)$ and that the initial data belong to $\E'(\R^n)$.
 
Typical examples are the wave equation
 \[
\partial_t^2u(t,x)-\sum_{i=1}^n a_i(t)\partial_{x_i}^2u(t,x)=f(t,x),
\]
where the coefficients $a_i$ are Heaviside functions or more in general equations of the type
\beq
\label{op_2}
D_t^2u(t,x)-\sum_{i=1}^n b_i(t)D_tD_{x_i}u(t,x)-\sum_{i=1}^n a_i(t)D_{x_i}^2u(t,x)=f(t,x),
\eeq
where the coefficients are bounded real valued functions with $a_i$ positive for all $i=1,\dots,n$ (see \cite{GR:14} for more details). Note that it is not restrictive to assume that the coefficients are compactly supported as in \cite{GR:14}.
An immediate higher order examples is given by the composition of a finite number of hyperbolic second order operators as in \eqref{op_2}, i.e.,
\[
\biggl(\Pi_{k=1}^m\biggl( D_t^2-\sum_{i=1}^n b_{k,i}(t)D_tD_{x_i}-\sum_{i=1}^n a_{k,i}(t)D_{x_i}^2\biggr)\biggr)u(t,x)
\]
plus lower order terms. Its characteristic polynomial  
\[
\Pi_{k=1}^m \biggl(\tau^2-\sum_{i=1}^n b_{k,i}(t)\tau\xi_i-\sum_{i=1}^n a_{k,i}(t)\xi_i^2\biggr).
\]
has $2m$ real roots.

Hyperbolic Cauchy problems with non regular coefficients naturally appear in applied sciences as geophysics and seismology,
to model  delta-like sources and discontinuous or more irregular media. We refer the reader to \cite{Marsan-Bean} and 
\cite{Hormann-de-Hoop:AAM-2001} for a survey on this kind of applications.

The Cauchy problem \eqref{intro_eq}-\eqref{intro_ic} has been extensively studied when the coefficients are at least H\"older. See the first work by Colombini and Kinoshita in one space dimension in \cite{ColKi:02}, the extension to any space dimension in \cite{GR:11} and the recent paper \cite{GR:12} for the treatment of lower order terms by Levi conditions.
In all these paper well-posedness is proven in Gevrey classes and by duality in ultradistributional spaces. Note that even if the coefficients are very regular ($C^\infty$) well-posedness has to be expected to hold only in Gevrey classes (see \cite{Colombini-deGiordi-Spagnolo-Pisa-1979} and \cite{CJS:Pisa-1983}) and the corresponding Cauchy problems might be distributionally ill-posed due to the presence of multiplicities (see the examples constructed in \cite{Colombini-Spagnolo:Acta-ex-weakly-hyp} and \cite{Colombini-Jannelli-Spagnolo:Annals-low-reg}). No well-posedness results are known when the assumption of H\"older regularity is dropped.

Our aim in this paper is to solve the Cauchy problem in \eqref{intro_eq}-\eqref{intro_ic}. Due to the low regularity of the equation's coefficients and characteristic roots, which does not allow classical Gevrey or ultradistributional solutions, we will look for \emph{very weak solutions}, namely for nets of solutions of the regularised problem obtained from \eqref{intro_eq}-\eqref{intro_ic} by convolution with Friedrichs mollifiers.  Inspired by the treatment of second order equations in \cite{GR:14}, our starting point is the regularisation of the roots and initial data. This is a technique quite common in hyperbolic equations which, under sufficient regularity assumptions, leads to a classical Gevrey well-posedness result by relating the regularising parameter with the phase variable at the Fourier transform level, as in \cite{ColKi:02} and \cite{GR:11}. In this paper we focus on the regularising nets and the corresponding nets of solutions, proving existence of a very weak solution and consistency with the classical Gevrey or ultradistributional solution whenever it exists.

\subsection{Basic notions and very weak solutions}\quad\vspace{0.2cm}

Before stating the definition of very weak solution we recall few preliminary notions concerning Gevrey functions and moderate nets. For more details we refer the reader to \cite{GR:14} and \cite{Rod:93}. 

Let $s\ge 1$. We say that $f\in C^\infty(\R^n)$ belongs to the Gevrey class $\gamma^s(\R^n)$ if for every compact set $K\subset\R^n$ there exists a constant $C>0$ such that for all $\alpha\in\N^n$ we have the estimate
\[
\sup_{x\in K}|\partial^\alpha f(x)|\le C^{|\alpha|+1}(\alpha!)^s.
\]
In this paper we make use of the following notion of moderate net.   
\begin{definition}
\label{def_mod_intro}
\leavevmode
\begin{itemize}
\item[(i)] A net of functions $(h_\eps)_\eps\in C^\infty(\R^n)^{(0,1])}$ is $C^\infty$-moderate if for all $K\Subset\R^n$ and for all $\alpha\in{\N}^n$ there exist $N\in\N$ and $c>0$ such that
\[
\sup_{x\in K}|\partial^\alpha h_\eps(x)|\le c\eps^{-N},
\]
for all $\eps\in(0,1]$.\\ 
\item[(ii)] A net of functions $(h_\eps)_\eps\in \gamma^s(\R^n)^{(0,1]}$ is $\gamma^s$-moderate if for all $K\Subset\R^n$ there exists a constant $c_K>0$ and there exists $N\in\N$ such that 
\[
|\partial^\alpha h_\eps(x)|\le c_K^{|\alpha|+1}(\alpha !)^s \eps^{-N-|\alpha|},
\]
for all $\alpha\in\N^n$, $x\in K$ and $\eps\in(0,1]$.\\
\item[(iii)] A net of functions $(h_\eps)_\eps\in C^\infty([0,T];\gamma^s(\R^n))^{(0,1]}$ is $C^\infty([0,T];\gamma^s(\R^n))$-mo\-de\-rate 
if for all $K\Subset\R^n$ there exist $N\in\N$, $c>0$ and, for all $k\in\N$ there exist $N_k>0$ and $c_k>0$ such that
\[
|\partial_t^k\partial^\alpha_x h_\eps(t,x)|\le c_k\eps^{-N_k} c^{|\alpha|+1}(\alpha !)^s \eps^{-N-|\alpha|},
\]
for all $\alpha\in\N^n$, for all $t\in[0,T]$, $x\in K$ and $\eps\in(0,1]$.
\end{itemize}
\end{definition}
More in general, given two spaces $X$ and $Y$, with $Y\subseteq X$ (usually $X=\D'(\R^n), \E'(\R^n),\\ L^\infty(\R^n),\dots$ and $Y=C^\infty(\R^n), \gamma^s(\R^n), \dots$), we use the expression regularisation of $h\in X$ for a net of regular functions $(h_\eps)_\eps\in Y^{(0,1]}$ approximating $h$ in $X$ as $\eps\to 0$. 

We are now ready to state the definition of \emph{very week solution} of the Cauchy problem 
\beq
\label{intro_CP}
\begin{split}
D^m_t u-\sum_{j=1}^m\sum_{|\nu|= j} a_{\nu,j}(t)D_t^{m-j}D^\nu_x u-\sum_{j=1}^m\sum_{|\nu|<j}b_{\nu,j}(t)D^{m-j}_tD^\nu_x u &=f(t,x),\, t\in[0,T], x\in\R^n,\\
D_t^{k}u(0,x)&=g_k,\quad k=0,\cdots,m-1.
\end{split}
\eeq

\begin{definition}
\label{def_vws}
Let $s\ge1$. The net $(u_\eps)_\eps\in C^\infty([0,T];\gamma^s(\R^n))$ is a very weak solution of order $s$ of the 
Cauchy problem \eqref{intro_CP} if there exist 
\begin{itemize}
\item[(i)] $C^\infty$-moderate regularisations $a_{\nu,j,\eps}$ and  $b_{\nu,j\eps}$ of all the coefficients  $a_{\nu,j}$ and $b_{\nu,j}$, respectively,
\item[(ii)] a $C^\infty([0,T];\gamma^s(\R^n))$-moderate regularisation $f_\eps$ of the right-hand side $f$, and
\item[(iii)] $\gamma^s$-moderate regularisations $g_{k,\eps}$ of the initial data $g_k$ for $k=0,\cdots m-1$, 
\end{itemize}
such that $(u_\eps)_\eps$ solves the regularised problem
\[
\begin{split}
D^m_t u-\sum_{j=1}^m\sum_{|\nu|= j} a_{\nu,j,\eps}(t)D_t^{m-j}D^\nu_x u-\sum_{j=1}^m\sum_{|\nu|<j}b_{\nu,j,\eps}(t)D^{m-j}_tD^\nu_x u &=f_\eps(t,x),\\
D_t^{k}u(0,x)&=g_{k,\eps}.
\end{split}
\]
for $t\in[0,T], x\in\R^n,  k=0,\cdots,m-1$ and $\eps\in(0,1]$, and is $C^\infty([0,T];\gamma^s(\R^n))$-moderate.
\end{definition}

\subsection{Paper's aim and main result}\quad\vspace{0.2cm}

The aim of this paper is to prove that the hyperbolic Cauchy problem \eqref{intro_CP} admits very weak solutions when the coefficients of the principal part are only bounded and the lower order terms, the right-hand side and the initial data are compactly supported distributions and that these weak solutions converge to the classical one in case of H\"older coefficients. 
Note that differently from the classical results for H\"older coefficients in \cite{ColKi:02, GR:11} we do not require here any technical hypothesis on the roots, namely the uniformity property (2.5) in \cite{GR:11}, and differently from \cite{GR:12} no Levi conditions are needed on the lower order terms. Some first results of existence of very weak solutions have been recently obtained in \cite{GR:14} for some family of homogeneous second order equations. In this paper we extend \cite{GR:14} to higher order equations and we deal with lower order terms as well.  

This is our main result:

\begin{theorem}
\label{main_theo}
 The hyperbolic Cauchy problem
\[
\begin{split}
D^m_t u-\sum_{j=1}^m\sum_{|\nu|= j} a_{\nu,j}(t)D_t^{m-j}D^\nu_x u-\sum_{j=1}^m\sum_{|\nu|<j}b_{\nu,j}(t)D^{m-j}_tD^\nu_x u &=f(t,x),\, t\in[0,T], x\in\R^n,\\
D_t^{k}u(0,x)&=g_k,\quad k=0,\cdots,m-1,
\end{split}
\]
where the equation coefficients are compactly supported in $t$, $a_{\nu,j}\in L^\infty([0,T])$ for $|\nu|=j$, $j=1,\dots,m$, $b_{\nu,j}\in\E'([0,T])$ for $|\nu|<j$, $j=1,\dots,m$, $f\in \E'([0,T])\otimes\E'(\R^n)$ and $g_k\in\E'(\R^n)$ for all $k=0,\dots,m-1$, has a very weak solution of order $s$ for any $s>1$.
\end{theorem} 

Since the Cauchy problem above is solved by reduction to a hyperbolic first order system in Sylvester form we automatically have that any hyperbolic first order system of size $m\times m$ in block Sylvester form,  with $t$-dependent bounded real eigenvalues (of the principal part), lower order terms in $\E'([0,T])$, right-hand side in $( \E'([0,T])\otimes\E'(\R^n))^m$ and initial data in $\E'(\R^n)^m$ has a very weak solution of order $s$ for any $s>1$.  

D'Ancona and Spagnolo proved in \cite[Section 4]{DS} that any hyperbolic system can be reduced to block Sylvester form. Therefore, by combining this result with the observation above we can state the following theorem.
\begin{theorem}
\label{main_theo_syst}
Any linear first order hyperbolic system of size $m\times m$ with compactly supported $t$-dependent coefficients and bounded eigenvalues with respect to $t\in[0,T]$, lower order terms in $\E'([0,T])$, right-hand side in $( \E'([0,T])\otimes\E'(\R))^m$ and initial data in $\E'(\R)^m$ has a very weak solution of order $s$ for any $s>1$.
\end{theorem}
Note that few results are known concerning the well-posedness of hyperbolic sy\-stems with multiplicities. Whenever no particular assumptions are made on the multiplicities a certain regularity of the coefficients is required. This means to work under the assumptions that the $t$-dependent coefficients are at least H\"older. See the work of  d'Ancona, Kinoshita and Spagnolo \cite{dAKS:04, dAKS:08} for $t$-dependent hyperbolic systems of size $2\times 2$ and $3\times 3$ and the extension to any size given by Yuzawa in \cite{Yu:05}. Theorem \ref{main_theo_syst} is the first result for hyperbolic systems with multiplicities which goes beyond the traditional hypothesis of  H\"older regularity and opens an exciting and new research path.   

This paper deals with scalar equations and systems with time dependent coefficients only. The dependence in $x$ is a rather problematic issue and so far has been treated under strong regularity hypotheses (Gevrey). We mention the foundational work of Bronshtein \cite{Bronshtein:TMMO-1980} and Nishitani \cite{Nishitani:BSM-1983} for scalar equations and the paper of Kajitani and Yuzawa \cite{KY:06} for systems. The relationship between lower $x$-regularity and very weak solvability will be analysed in a future paper.

For the sake of the reader we conclude this introduction with the Fourier characterisations of Gevrey functions, ultradistributions and moderate nets which will be heavily used throughout the paper.

\subsection{Fourier characterisations}\quad\vspace{0.2cm}

Let $\gamma^s_c(\R^n)$ denote the space of compactly supported Gevrey functions of order $s$ and let $\lara{\xi}=(1+|\xi|^2)^{\frac{1}{2}}$. The proof of the following proposition can be found in  \cite[Theorem 1.6.1]{Rod:93}).
\begin{proposition}
\label{prop_Fourier}
\leavevmode
\begin{itemize}
\item[(i)] Let $u\in \gamma^s_c(\R^n)$. Then, there exist constants $c>0$ and $\delta>0$ such that
\beq
\label{fou_gey_1}
|\widehat{u}(\xi)|\le c\,\esp^{-\delta\lara{\xi}^{\frac{1}{s}}}
\eeq
for all $\xi\in\R^n$.
\item[(ii)] Let $u\in\S'(\R^n)$. If there exist constants $c>0$ and $\delta>0$ such that \eqref{fou_gey_1} holds then $u\in \gamma^s(\R^n)$.
  \end{itemize}
\end{proposition}
Gevrey-moderate nets can be characterised at the Fourier transform level as well. 
\begin{proposition} 
\label{prop_reg_gevrey_mod}
\leavevmode
\begin{itemize}
\item[(i)] If $(h_\eps)_\eps$ is $\gamma^s$-moderate and there exists $K\Subset\R^n$ such that $\supp\, u_\eps\subseteq K$ for all $\eps\in(0,1]$ then there exist $c,c'>0$ and $N\in\N$ such that
\beq
\label{star1}
|\widehat{h_\eps}(\xi)|\le c'\eps^{-N}\esp^{-c\eps^{\frac{1}{s}}\lara{\xi}^{\frac{1}{s}}},
\eeq
for all $\xi\in\R^n$ and $\eps\in(0,1]$.
 \item[(ii)] If $(h_\eps)_\eps$ is a net of tempered distributions with $(\widehat{u_\eps})_\eps$ satisfying \eqref{star1} then $(h_\eps)_\eps$ is $\gamma^s$-moderate.
\end{itemize}
\end{proposition}
For a detailed proof of Proposition \ref{prop_reg_gevrey_mod} see \cite[Proposition 4.3]{GR:14}.

In this paper we will also make use of the spaces $\D'_{(s)}(\R^n)$ and $\E'_{(s)}(\R^n)$ of (Gevrey Beurling) ultradistributions and compactly supported ultradistributions, respectively. $\D'_{(s)}(\R^n)$ is the dual of the space $\gamma^{(s)}_c(\R^n)$ of compactly supported Gevrey Beurling functions. Recall that for $s\ge 1$, $f\in \Cinf(\R^n)$ belongs to $\gamma^{(s)}(\R^n)$ if for every compact set $K\subset\R^n$ and for every constant $A>0$ there exists a constant $C_{A,K}>0$ such that for all $\alpha\in{\N}^n$ the estimate
\beq
\label{GB_fun}
\sup_{x\in K}|\partial^\alpha f(x)|\le C_{A,K} A^{|\alpha|}(\alpha!)^s
\eeq
holds. 

In analogy with Gevrey functions, ultradistributions can be characterised by Fourier transform. This means that if $v\in \mathcal{E}'_{(s)}(\R^n)$ then there exist $\nu>0$ and $C>0$ such that
\beq
\label{ft_ultra}
|\widehat{v}(\xi)|\le C\,\esp^{\nu\lara{\xi}^{\frac{1}{s}}}
\eeq
for all $\xi\in\Rn$, and if a real analytic functional $v$ satisfies \eqref{ft_ultra} then $v\in\D'_{(s)}(\R^n)$.

\section{Regularisation and reduction to a system}
In this section we show how to regularise the Cauchy problem \eqref{intro_CP}. We start by analysing the coefficients of the principal part then we pass to the lower order terms, the right-hand side and the initial data.

We work under the assumption that the $m$ real roots $\lambda_j(t,\xi)$ of the equation
\[
D^m_t u-\sum_{j=1}^m\sum_{|\nu|= j} a_{\nu,j}(t)D_t^{m-j}D^\nu_x u-\sum_{j=1}^m\sum_{|\nu|<j} b_{\nu,j}(t)D^{m-j}_tD^\nu_x u=0
\]
are compactly supported and bounded in $t\in[0,T]$ and homogeneous of order $1$ in $\xi$, i.e. there exists $c>0$ such that
\[
|\lambda_{j}(t,\xi)|\le c|\xi|,
\] 
for all $t\in[0,T]$ and $\xi\neq 0$. Note that the dependence in $\xi$ is continuous and that the boundedness of the roots $\lambda_j$ forces the coefficients $a_{\nu,j}$ of the principal part to be bounded as well. We are quite general in the choice of the lower order terms, in the sense that we take distributions with compact support contained in $[0,T]$. It follows that $b_{\nu,j}\in\E'(\R)$ for $j=1,\dots,m$ and $|\nu|< j$.

Let $\varphi$ be a mollifier, i.e, $\varphi\in C^\infty_c(\R)$ with $\int\varphi=1$ and $\varphi_\eps(t)=\eps^{-1}\varphi(t/\eps)$. By convolution with the mollifier $\varphi_\eps$ we can regularise the roots $
\lambda_j\in L^\infty(\R)$ obtaining $m$ nets 
\beq
\label{moll_eigen}
\lambda_{j,\eps}(t,\xi)=(\lambda_j(\cdot,\xi)\ast\varphi_\eps)(t)
\eeq
 fulfilling the following property: for all $j=1,\dots,m$ and for all $k\in\N$ there exist $c>0$ such that
\beq
\label{lambda_eps}
|d^{(k)}_t\lambda_{j,\eps}(t,\xi)|\le c\eps^{-k}|\xi|,
\eeq
for all $t\in[0,T]$, $\xi\neq 0$ and $\eps\in(0,1]$. By setting
\beq
\label{charac_reg}
\tau^m-\sum_{j=1}^m\sum_{|\nu|= j} a_{\nu,j,\eps}(t)\tau^{m-j}\xi^\nu=\Pi_{j=1,\dots,m}(\tau-\lambda_{j,\eps}(t,\xi))
\eeq
we can find a way to approximate the equation's coefficients $a_{\nu,j}(t)$ which is  regular in $t$, i.e., $C^\infty$. In the sequel we will make use of the notation 
\[
\sigma_h^{(m)}(\lambda_\eps)=(-1)^h\sum_{1\le i_1<...<i_h\le m}\lambda_{i_1,\eps}...\lambda_{i_h,\eps},
\]
introduced in \cite{GR:12}, where $\lambda_\eps=(\lambda_{1,\eps},\lambda_{2,\eps},\dots,\lambda_{m,\eps})$, $h=0,\dots,m$ with $\sigma_0^{(m)}(\lambda_\eps)=1$. In this way we can write the right-hand side of \eqref{charac_reg} as 
\[
\sum_{j=0}^m \tau^{m-j}\sigma^{(m)}_j(\lambda_\eps).
\]
and conclude that
\beq
\label{formula_coeff}
-\sum_{|\nu|= j} a_{\nu,j,\eps}(t)\xi^\nu=\sigma_j^{(m)}(\lambda_\eps),
\eeq
for all $j=1,\dots,m$. More precisely we have the following proposition.
\begin{proposition}
\label{prop_eigen_reg}
Let the $m$ real roots $\lambda_{j}(t,\xi)$ of the equation \eqref{intro_eq} be compactly supported and bounded in $t\in[0,T]$.
Then every net $a_{\nu,j,\eps}$, $j=1,\dots,m$, defined by \eqref{formula_coeff} is $C^\infty$-moderate, in the sense that for all $k\in\N$ there exist $N\in\N$ and $c>0$ such that
\[
|d^{(k)}_t a_{\nu,j,\eps}(t)|\le c\,\eps^{-N},
\]
for all $t\in[0,T]$ and $\eps\in(0,1]$, 
and converges to $a_{\nu,j}$ in $L^\infty(\R)$ as $\eps\to 0$.

\end{proposition}
\begin{proof}
First we prove that the nets  $a_{\nu,j,\eps}(t)$, $j=1,\dots,m$, are moderate. 

We begin by observing that since the nets $\lambda_{j,\eps}$, $j=1,\dots,m$ are moderate in the sense of \eqref{lambda_eps} then every $\sigma_j(\lambda_\eps)$ is moderate as well, i.e., for all $j=1,\dots,m$ and for all $k\in\N$ there exist $N\in\N$ and $c>0$ such that
\[
|d^{(k)}_t\sigma_j(\lambda_{\eps})(t,\xi)|\le c\eps^{-N}|\xi|^j,
\]
for all $t\in[0,T]$, $\xi\neq 0$ and $\eps\in(0,1]$. Making now use of \eqref{formula_coeff} we can prove by induction on $|\nu|$ that the coefficients $a_{\nu,j,\eps}$ of the characteristic polynomial are moderate as well. Indeed, if $|\nu|=1$ from  \eqref{formula_coeff} we get 
\[
-\sum_{|\nu|=1} a_{\nu,1,\eps}(t)\xi^\nu=\sigma_1^{(m)}(\lambda_\eps)=-\sum_{j=1}^m \lambda_{j,\eps}(t,\xi),
\]
which can be rewritten as
\[
\sum_{k=1}^n a_{\nu_k,1,\eps}(t)\xi_k=\sum_{j=1}^m \lambda_{j,\eps}(t,\xi),
\]
with $\nu_k$ $n$-index with $k$-entry equal to 1 and otherwise $0$. Let $e_k\in\R^n$ with $k$-entry equal to 1 and all the others 0. From the previous formula and \eqref{lambda_eps} we obtain 
\[
a_{\nu_k,1,\eps}(t)=\sum_{j=1}^m \lambda_{j,\eps}(t,e_k)
\]
which proves that the coefficients $a_{\nu_k,1,\eps}(t)$ are moderate for $k=1,\dots,n$. Assume now to have proved that the coefficients $a_{\nu,j,\eps}(t)$ are moderate for $|\nu|=j$ and $j=1,\dots,m-1$. We want to prove that $a_{\nu,m,\eps}(t)$ is moderate for $|\nu|=m$. From \eqref{formula_coeff} we have that
\[
-\sum_{|\nu|= m} a_{\nu,m,\eps}(t)\xi^\nu=\sigma_m^{(m)}(\lambda_\eps)=(-1)^m\lambda_{1,\eps}(t,\xi)\lambda_{2,\eps}(t,\xi)\cdots\lambda_{m,\eps}(t,\xi).
\]
By writing the previous formula as
\begin{multline}
\label{formula_imp}
-\sum_{\nu\in I_{n-1}} a_{\nu,m,\eps}(t)\xi^\nu-\sum_{\nu\in I_{n-2}} a_{\nu,m,\eps}(t)\xi^\nu-\cdots -\sum_{\nu\in I_0} a_{\nu,m,\eps}(t)\xi^\nu\\
=(-1)^m\lambda_{1,\eps}(t,\xi)\lambda_{2,\eps}(t,\xi)\cdots\lambda_{m,\eps}(t,\xi),
\end{multline}
where $I_j$ is the set of all multi-indexes $\nu$ with length $m$ having $j$ 0-entries, $j=0,\dots n-1$, we easily see that the coefficients $a_{\nu,m,\eps}(t)$ with $\nu\in I_{n-1}$ are moderate. Indeed, for $k=1,\dots,n$, $\nu_k$ multi-index with $k$-entry equal to $m$ and all the others $0$ and $e_k\in\R^n$ defined above, we obtain
\beq
\label{formula_1}
-a_{\nu_k,m,\eps}(t)=(-1)^m\lambda_{1,\eps}(t,e_k)\lambda_{2,\eps}(t,e_k)\cdots\lambda_{m,\eps}(t,e_k).
\eeq
In other words by suitably choosing the points $e_k$ we have found an invertible $n\times n$-matrix $A$ and $n$-moderate nets $r_{1,\eps}(t),\dots r_{n,\eps}(t)$ such that 
\beq
\label{formula_A}
A
                      \left(\begin{array}{c}
                              a_{\nu_1,m,\eps}(t) \\
                                a_{\nu_2,m,\eps}(t)\\
                                \dots\\
                                a_{\nu_n,m,\eps}(t)
                             \end{array}
                           \right)=
                           \left(\begin{array}{c}
                              r_{1,\eps}(t) \\
                                r_{2,\eps}(t)\\
                                \dots\\
                                r_{n,\eps}(t)
                             \end{array}
                           \right)
\eeq
We now bring $-\sum_{\nu\in I_{n-1}} a_{\nu,m,\eps}(t)\xi^\nu$ on the right-hand side of \eqref{formula_imp} and we focus on $\sum_{\nu\in I_{n-2}} a_{\nu,m,\eps}(t)\xi^\nu$. Note that by what we have just proved the new right-hand side will be moderate in $t$. Let $l$ be the number of the elements of $I_{n-2}$. Arguing as above, by suitably choosing a finite number of points $\xi$ we can generate an invertible $l\times l$-matrix $B$ and $l$-moderate nets $s_{1,\eps}(t),\dots s_{l,\eps}(t)$ such that
\beq
\label{formula_B}
B
\left(\begin{array}{c}
                              a_{\nu'_1,m,\eps}(t) \\
                                a_{\nu'_2,m,\eps}(t)\\
                                \dots\\
                                a_{\nu'_l,m,\eps}(t)
                             \end{array}
                           \right)=
                           \left(\begin{array}{c}
                              s_{1,\eps}(t) \\
                                s_{2,\eps}(t)\\
                                \dots\\
                                s_{l,\eps}(t)
                             \end{array}
                           \right),
\eeq
where $\nu'_1,\nu'_2\dots,\nu'_l$ are the elements of $I_{n-2}$.  It follows that all the coefficients $a_{\nu,m,\eps}(t)$ with $\nu\in I_{n-2}$ are moderate nets. We can now also bring -$\sum_{\nu\in I_{n-2}} a_{\nu,m,\eps}(t)\xi^\nu$ to the right-hand side of \eqref{formula_imp} and conclude that
\[
-\sum_{\nu\in I_{n-3}} a_{\nu,m,\eps}(t)\xi^\nu-\cdots -\sum_{\nu\in I_0} a_{\nu,m,\eps}(t)\xi^\nu
\]
is equal to a $t$-moderate net. By iterating the previous argument a finite number of times we conclude that all the coefficients $a_{\nu,m,\eps}(t)$ with $\nu\in I_j$, $j=n-3, \dots, 1, 0$ are moderate too. 

The proof by induction above can be also used to prove that the coefficients $a_{\nu,m,\eps}(t)$ converge to $a_{\nu,m}$ in $L^\infty(\R)$. Recall that by construction (regularisation with a mollifier) $\lambda_{j,\eps}\to \lambda_j$ in $L^\infty(\R)$ for all $j=1,\dots,m$. Take now $|\nu|=1$ in \eqref{formula_coeff}. From
\[
a_{\nu_k,1,\eps}(t)=\sum_{j=1}^m \lambda_{j,\eps}(t,e_k), 
\]
where $\nu_k$ and $e_k$ are the $n$-index and the element of $\R^n$ respectively, with $k$-entry equal to 1 and otherwise $0$, $k=1,\dots,n$, we easily see that 
\[
a_{\nu_k,1,\eps}(t)=\sum_{j=1}^m \lambda_{j,\eps}(t,e_k)\to \sum_{j=1}^m \lambda_{j}(t,e_k)=a_{\nu_k,1}(t)
\]
in $L^\infty(\R)$ as $\eps\to 0$. This proves that $a_{\nu,1,\eps}$ is a regularisation of $a_{\nu,1}$ when $|\nu|=1$. Assume now that $a_{\nu,j,\eps}\to a_{\nu,j}$ in $L^\infty(\R)$ for $|\nu|=j$ and $j\le m-1$. We still need to prove that $a_{\nu,m,\eps}\to a_{\nu,m}$ in $L^\infty(\R)$ for $|\nu|=m$. As in the first part of the proof we set
\[
\{\nu: |\nu|=m\}=\cup_{j=n-1,n-2,\dots,0}\,I_j
\]
where $I_j$ is the set of all multi-indexes with length $m$ having $j$ 0-entries. From \eqref{formula_1} or equivalently the matrix expression \eqref{formula_A} we have that 
\begin{multline*}
-a_{\nu_k,m,\eps}(t)=(-1)^m\lambda_{1,\eps}(t,e_k)\lambda_{2,\eps}(t,e_k)\cdots\lambda_{m,\eps}(t,e_k)\\
\to (-1)^m\lambda_{1}(t,e_k)\lambda_{2}(t,e_k)\cdots\lambda_{m}(t,e_k)= -a_{\nu_k,m}(t)
\end{multline*}
in $L^\infty(\R)$, for any $\nu_k$ defined above. Note that in \eqref{formula_B} the invertible matrix $B$ does not depend on $\eps$ while the right-hand side can be seen as a matrix-valued function $F$ independent on $\eps$ which depends continuously on $a_{\nu,m,\eps}$ with $\nu\in I_{n-1}$ and the regularised roots $\lambda_{j,\eps}$. It follows that $a_{\nu,m,\eps}\to a_{\nu,m}$ in $L^\infty(\R)$ for every $\nu\in I_{n-2}$. By iterating the same argument a finite number of times we conclude that $a_{\nu,m,\eps}\to a_{\nu,m}$ in $L^\infty(\R)$ for all $\nu\in I_j$ with $j=n-2,\dots,1,0$. 
\end{proof}
\begin{remark}
\leavevmode
\begin{trivlist}
\item[(i)] Adopting the language of Definition \ref{def_mod_intro} we  can summarise Proposition \ref{prop_eigen_reg} as follows: there exists a $C^\infty$-moderate regularisation of the coefficients $a_{\nu,j}$ of the principal part. 
\item[(ii)]
Note that Proposition \ref{prop_eigen_reg} holds for any $C^\infty$-moderate regularisation of the roots $\lambda_j$, i.e., for any net $\lambda_{j,\eps}(\cdot,\xi)$ converging to $\lambda_j(\cdot,\xi)$ in $L^\infty$  as $\eps\to 0$ such that \eqref{lambda_eps} holds. In particular in this paper instead of $\eps$ we will use a general net $\omega(\eps)\to 0$ and a regularisation which separates the roots as well, i.e., we will set
\beq
\label{eigen_reg_omega}
\lambda_{j,\eps}(t,\xi)=(\lambda_j(\cdot,\xi)\ast\varphi_{\omega(\eps)})(t)+j\omega(\eps)\lara{\xi},
\eeq
with $\omega(\eps)\ge c\eps^{r}$ for some $c,r>0$ and $\varphi$ mollifier as in \eqref{moll_eigen}. The corresponding regularised operator
\[
D^m_t-\sum_{j=1}^m\sum_{|\nu|= j} a_{\nu,j,\eps}(t)D_j^{m-j}D^\nu_x 
\]
is therefore strictly hyperbolic with smooth coefficients whereas the original operator might be weakly hyperbolic. 
\end{trivlist}
\end{remark}

We now pass to consider the lower order terms, the initial data and the right-hand side of the Cauchy problem \eqref{intro_CP}.

\subsection{Regularisation of the lower order terms}

The lower order terms $b_{\nu,j}\in\E'(\R)$ can be easily regularised via convolution with a mollifier $\varphi_{\omega(\eps)}$ as the one used for the eigenvalues $\lambda_j$ in \eqref{eigen_reg_omega}. Indeed, by the structure theorem of compactly supported distributions and dominated convergence we have that the net
\[
b_{\nu,j,\eps}(t)=(b_j\ast\varphi_{\omega(\eps)})(t)
\]
is $C^\infty$-moderate and converges to $b_{\nu,j}$ in $\E'(\R)$ as $\eps\to 0$.

\subsection{Regularisation of the initial data}

Before proceeding with this regularisation, it is useful to recall a few results proven in \cite{GR:14} which will allow us to find suitable regularisations.

First we introduce the Gelfand-Shilov space $\mathcal{S}^{(s)}(\R^n)$, $s>1$, of all $f\in\Cinf(\R^n)$ such that
\[
\Vert f\Vert_{b,s}=\sup_{\alpha,\beta\in\N^n}\int_{\R^n}\frac{|x^\beta|}{b^{|\alpha+\beta|}\alpha !^s \beta !^s}|\partial^\alpha f(x)|\, dx<\infty
\]
for all $b>0$.

Since $\mathcal{S}^{(s)}(\R^n)$ is Fourier transform invariant (see e.g. \cite[Chapter 6]{NiRo:10} and \cite{Teo:06}) it follows that the inverse Fourier transform $\phi=\mF^{-1}\psi$ of a function $\psi\in\mathcal{S}^{(s)}(\R^n)$ identically $1$ in a neighborhood of $0$ is a function $\phi\in\mathcal{S}^{(s)}(\R^n)$ with
\beq
\label{mollifier}
\int\phi(x)\, dx=1,\qquad \text{and}\qquad \int x^\alpha\phi(x)\, dx=0,\quad \text{for all $\alpha\neq 0$.}
\eeq
For instance, one can take $\psi\in \gamma^{(s)}(\R^n)\cap \Cinfc(\R^n)$, where $\gamma^{(s)}(\R^n)$ is the space of Gevrey Beurling functions defined in \eqref{GB_fun}. 

We say that $\phi\in\mathcal{S}^{(s)}(\R^n)$ is a mollifier if the property \eqref{mollifier} holds. Finally, as in as in \cite{BenBou:09} let us define
\beq
\label{def_rho}
\rho_{\omega(\eps)}(x):=\omega(\eps)^{-n}\phi\biggl(\frac{x}{\omega(\eps)}\biggr)\chi(x|\log
\omega(\eps)|),
\eeq
where $\chi\in \gamma^s(\R^n)$ with $0\le\chi\le 1$, $\chi(x)=0$ for $|x|\ge 2$ and $\chi(x)=1$ for $|x|\le 2$. By construction this is a net of Gevrey functions of order $s$.
  
Proposition 6.1 in \cite{GR:14} investigates the convolution of compactly supported distributions with the mollifier $\rho_{\omega(\eps)}$ as follows.
\begin{proposition}
 \label{prop_reg_distr}
Let $u\in\E'(\R^n)$ and $\rho_{\omega(\eps)}$ as in \eqref{def_rho} . Then, there exists $K\Subset\R^n$ such that $\supp(u\ast\rho_{\omega(\eps)})\subseteq K$ for all $\eps$ small enough and there exist $C>0$, $N\in\N$ and $\eta\in(0,1]$ such that
\[
|\partial^\alpha(u\ast\rho_{\omega(\eps)})(x)|\le C^{|\alpha|+1} (\alpha !)^s \omega(\eps)^{-|\alpha|-N}
\]
for all $\alpha\in\N^n$, $x\in\R^n$ and $\eps\in(0,\eta]$.
\end{proposition}
Note that it is not restrictive to set $\eta=1$ in Proposition \ref{prop_reg_distr} and that $(u\ast\rho_{\omega(\eps)})_\eps$ is a $\gamma^s$-moderate regularisation of $u$. Indeed, the $\gamma^s$-moderateness is stated above and by the structure theorem for compactly supported distributions and dominated convergence $u\ast\rho_{\omega(\eps)}\to u$ in $\E'(\R^n)$ as $\eps\to 0$. This is the kind of regularisation that we use for the initial data $g_k$ in this paper. 

\subsection{Regularisation of the right-hand side}
The right-hand side $f$ is regularised with a usual mollifier $\varphi_{\omega(\eps)}$ in the variable $t$ and with a mollifier $\rho_{\omega(\eps)}$ as above in the variable $x$. More precisely, we set
\[
f_\eps(t,x)=(f\ast (\varphi_{\omega(\eps)}\otimes\rho_{\omega(\eps)}))(t,x)
\]
and we get a $C^\infty$-moderate net with respect to $t$ and a $\gamma^s$-moderate net with respect to $x$.

\subsection{Conclusion}

We regularise the Cauchy problem 
\[
\begin{split}
D^m_t u-\sum_{j=1}^m\sum_{|\nu|= j} a_{\nu,j}(t)D_t^{m-j}D^\nu_x u-\sum_{j=1}^m\sum_{|\nu|<j} b_{\nu,j}(t)D^{m-j}_tD^\nu_x u &=f(t,x),\quad t\in[0,T], x\in\R^n,\\
D_t^{k}u(0,x)&=g_k,\quad k=0,\cdots,m-1.
\end{split}
\]
by setting
\[
\begin{split}
D^m_t u-\sum_{j=1}^m\sum_{|\nu|= j} a_{\nu,j,\eps}(t)D_t^{m-j}D^\nu_x u-\sum_{j=1}^m\sum_{|\nu|<j} b_{\nu,j,\eps}(t)D^{m-j}_tD^\nu_x u &=f_\eps(t,x),\\ 
D_t^{k}u(0,x)&=g_{k,\eps},
\end{split}
\]
where
\begin{itemize}
\item[(i)] $a_{\nu,j,\eps}$ is a $C^\infty$-moderate regularisation of $a_{\nu,j}$ for $|\nu|= j$ and $j=1,\dots,m$ obtained through
\[
\lambda_{j,\eps}(t,\xi)=(\lambda_j(\cdot,\xi)\ast\varphi_{\omega(\eps)})(t)+j\omega(\eps)\lara{\xi},
\]
$j=1,\dots,m$,
\item[(ii)] $b_{\nu,j,\eps}$  is a $C^\infty$-moderate regularisation of $b_{\nu,j}$ for $|\nu|< j$ and $j=1,\dots,m$, of the type
\[
b_{\nu,j,\eps}(t)=(b_{\nu,j}\ast\varphi_{\omega(\eps)})(t),
\]
\item[(iii)] $f_\eps$ is a $C^\infty([0,T];\gamma^s(\R^n))$-moderate regularisation of $f$, of the type
\[
f_\eps(t,x)=(f\ast (\varphi_{\omega(\eps)}\otimes\rho_{\omega(\eps)}))(t,x),
\]
\item[(iii)] $g_{k,\eps}$ is a $\gamma^s$-moderate regularisations of $g_k$ for $k=0,\dots,m-1$, of the type
\[
g_{k,\eps}=g\ast\rho_{\omega(\eps)}.
\]
\end{itemize}
More precisely, $a_{\nu,j,\eps}\to a_{\nu,j}$ in $L^\infty(\R)$ as $\eps\to 0$, $b_{\nu,j,\eps}\to b_{\nu,j}$ in $\E'(\R)$,  $g_{k,\eps}\to g_k$ in $\E'(\R^n)$ and $f_\eps\to f$ in $\E '(\R)\otimes \E'(\R^n)$ as $\eps\to 0$. The regularised Cauchy problem obtained in this way is strictly hyperbolic and has smooth coefficients. The net $\omega(\eps)$ will be suitably chosen in the sequel to guarantee the existence of a very weak solution.

\section{Very weak solutions}
In this section we prove that the Cauchy problem \eqref{intro_CP} admits very weak solutions, i.e. we prove Theorem \ref{main_theo}:\\
  
\emph{The hyperbolic Cauchy problem
\[
\begin{split}
D^m_t u-\sum_{j=1}^m\sum_{|\nu|= j} a_{\nu,j}(t)D_t^{m-j}D^\nu_x u-\sum_{j=1}^m\sum_{|\nu|<j}b_{\nu,j}(t)D^{m-j}_tD^\nu_x u &=f(t,x),\, t\in[0,T], x\in\R^n,\\
D_t^{k}u(0,x)&=g_k,\quad k=0,\cdots,m-1,
\end{split}
\]
where the equation coefficients are compactly supported in $t$, $a_{\nu,j}\in L^\infty([0,T])$ for $|\nu|=j$, $j=1,\dots,m$, $b_{\nu,j}\in\E'([0,T])$ for $|\nu|<j$, $j=1,\dots,m$, $f\in \E'([0,T])\otimes\E'(\R^n)$ and $g_k\in\E'(\R^n)$ for all $k=0,\dots,m-1$, has a very weak solution of order $s$ for any $s>1$.}\\
 
Theorem \ref{main_theo} is proven by energy estimates after reduction to a first order system of pseudodifferential equations.

\subsection{Reduction to system and energy estimates}

We perform a reduction to a first order system as in \cite{GR:11}. Let  $\lara{D_x}$ be
the pseudo-differential operator with symbol $\lara{\xi}$. The transformation
\[
u_k=D_t^{k-1}\lara{D_x}^{m-k}u,
\]
with $k=1,...,m$, makes the Cauchy problem 
\beq
\label{CP_reg}
\begin{split}
D^m_t u-\sum_{j=1}^m\sum_{|\nu|= j} a_{\nu,j,\eps}(t)D_t^{m-j}D^\nu_x u-\sum_{j=1}^m\sum_{|\nu|<j} b_{\nu,j,\eps}(t)D^{m-j}_tD^\nu_x u &=f_\eps(t,x),\\
D_t^{k}u(0,x)&=g_{k,\eps},\, k=0,\cdots,m-1,
\end{split} 
\eeq
equivalent to the following system

\beq
\label{syst_Taylor}
D_t\left(
                             \begin{array}{c}
                               u_1 \\
                               \cdot \\
                               \cdot\\
                               u_m \\
                             \end{array}
                           \right)
= \left(
    \begin{array}{ccccc}
      0 & \lara{D_x} & 0 & \dots & 0\\
      0 & 0 & \lara{D_x} & \dots & 0 \\
      \dots & \dots & \dots & \dots & \lara{D_x} \\
      l_{1,\eps} & l_{2,\eps} & \dots & \dots & l_{m,\eps} \\
    \end{array}
  \right)
  \left(\begin{array}{c}
                               u_1 \\
                               \cdot \\
                               \cdot\\
                               u_m \\
                             \end{array}
                           \right)
                           +
\left(\begin{array}{c}
                               0 \\
                               0 \\
                               \cdot\\
                               f_\eps \\
                             \end{array}
                           \right),
\eeq
where
\[
l_{j,\eps}=\sum_{|\nu|=m-j+1}a_{\nu,j,\eps}(t)D_x^\nu\lara{D_x}^{j-m}+\sum_{|\nu|=m-j}b_{\nu,j,\eps}(t)D_x^\nu\lara{D_x}^{j-m},
\] 
with initial condition
\beq
\label{ic_Taylor}
u_k|_{t=0}=\lara{D_x}^{m-k}g_{k-1,\eps},\qquad k=1,...,m.
\eeq
Let us denote the principal part $\sum_{|\nu|=m-j+1}a_{\nu,j,\eps}(t)D_x^\nu\lara{D_x}^{j-m}$of $l_{j,\eps}$ with $l_{(j,\eps)}$. Hence, the matrix in \eqref{syst_Taylor} can be written as $A+B$, with
\[
A=\left(
    \begin{array}{ccccc}
      0 & \lara{D_x} & 0 & \dots & 0\\
      0 & 0 & \lara{D_x} & \dots & 0 \\
      \dots & \dots & \dots & \dots & \lara{D_x} \\
      l_{(1,\eps)} & l_{(2,\eps)} & \dots & \dots & l_{(m,\eps)} \\
    \end{array}
  \right),
\]
and
\[
B=\left(
    \begin{array}{ccccc}
      0 & 0 & 0 & \dots & 0\\
      0 & 0 & 0& \dots & 0 \\
      \dots & \dots & \dots & \dots & 0 \\
      l_{1,\eps}-l_{(1,\eps)} & l_{2,\eps}-l_{(2,\eps)} & \dots & \dots & l_{m,\eps}-l_{(m,\eps)} \\
    \end{array}
  \right).
\]
By construction the roots $\lambda_{j,\eps}$ of the equation in \eqref{CP_reg} are the eigenvalues of the matrix $A$.
 
By Fourier transforming both sides of \eqref{syst_Taylor} we obtain the system
\beq
\label{system_new}
\begin{split}
D_t V&=A(t,\xi)V+B(t,\xi)V+\widehat{F}(t,\xi),\\
V|_{t=0}(\xi)&=V_0(\xi),
\end{split}
\eeq
where $V$ is the $m$-column with entries $v_k=\widehat{u}_k$, $V_0$ is the $m$-column with entries
$v_{0,k}=\lara{\xi}^{m-k}\widehat{g}_{k-1,\eps}$ and
\[
A(t,\xi)=\left(
    \begin{array}{ccccc}
      0 & \lara{\xi} & 0 & \dots & 0\\
      0 & 0 & \lara{\xi} & \dots & 0 \\
      \dots & \dots & \dots & \dots & \lara{\xi} \\
      l_{(1,\eps)}(t,\xi) & l_{(2,\eps)}(t,\xi) & \dots & \dots & l_{(m,\eps)}(t,\xi) \\
    \end{array}
  \right), 
\]
\[
B(t,\xi)=\left(
    \begin{array}{ccccc}
      0 & 0 & 0 & \dots & 0\\
      0 & 0 & 0& \dots & 0 \\
      \dots & \dots & \dots & \dots & 0 \\
      (l_{1,\eps}-l_{(1,\eps)})(t,\xi) & \dots & \dots & \dots & (l_{m,\eps}-l_{(m,\eps)})(t,\xi) \\
    \end{array}
  \right), 
\]
and
\[
\widehat{F}(t,\xi)=\left(\begin{array}{c}
                               0 \\
                               0 \\
                               \vdots\\
                               {\widehat{f}}_\eps(t,\cdot)(\xi) \\
                             \end{array}
                           \right).
\]
In the sequel we will focus on the system \eqref{system_new}. This is a strictly hyperbolic system with smooth coefficients. So, by the well-posedness results proven in \cite{GR:11} (Remark 8) we know that for all $\eps\in(0,1]$ and for all $s>1$ the Cauchy problem \eqref{system_new} has a net of solutions $V_\eps$ which, by inverse Fourier transform and by the transformation $u_k=D_t^{k-1}\lara{D_x}^{m-k}u$ above, generates the solution  $u_\eps\in C^\infty([0,T], \gamma^s(\R^n))$ to the Cauchy problem \eqref{CP_reg}. It is our task to prove that the net of solutions $u_\eps$ is $C^\infty([0,T], \gamma^s(\R^n))$-moderate. This will allow us to conclude that the Cauchy problem \eqref{intro_CP} admits a very weak solution of order $s$ and to prove Theorem \ref{main_theo}.

\subsection{Symmetriser}
The existence result in Theorem \ref{main_theo} will be deduced via energy estimates. The energy will be defined using the symmetriser of the (strictly) hyperbolic matrix
\[
 A\lara{\xi}^{-1}=\left(
    \begin{array}{ccccc}
      0 & 1 & 0 & \dots & 0\\
      0 & 0 & 1 & \dots & 0 \\
      \dots & \dots & \dots & \dots & 1 \\
      l_{(1,\eps)}(t,\xi)\lara{\xi}^{-1} & l_{(2,\eps)}(t,\xi)\lara{\xi}^{-1} & \dots & \dots & l_{(m,\eps)}(t,\xi)\lara{\xi}^{-1} \\
    \end{array}
  \right)
\]
In the sequel we collect some basic facts concerning symmetrisers mainly obtained adapting the corresponding results in \cite{J:09, JT}. We begin by stating that since the matrix $A\lara{\xi}^{-1}$ is hyperbolic with eigenvalues 
\[
\lambda_{1,\eps}(t,\xi)\lara{\xi}^{-1}\le\lambda_{2,\eps}(t,\xi)\lara{\xi}^{-1}\le\cdots\le\lambda_{m,\eps}(t,\xi)\lara{\xi}^{-1},
\]
(it is not restrictive to assume that the roots $\lambda_j$, $j=1,\dots,m$, are ordered and therefore their regularisations are ordered as well) then there exists a real symmetric $m\times m$-matrix $S(t,\xi)$ with $0$-order entries such that $SA-A^\ast S=0$ and 
\beq
\label{detS}
{\rm det}S(t,\xi)=\prod_{1\le j<i\le m}(\lambda_{i,\eps}(t,\xi)-\lambda_{j,\eps}(t,\xi))^2\lara{\xi}^{-2}
\eeq
(see \cite{J:09}). In particular, we have the following lemma (Lemma 2.1 in \cite{JT}). \begin{lemma}
\label{lemmaJT}
Let $N(t)$ be any symmetric positive semi-definite matrix with bounded coefficients on an interval $[a,b]$. Then, there exist two positive constants $c_1$ and $c_2$ depending only on the $L^\infty$-norm of the entries of $N(t)$ such that
\[
c_1\det N(t)|V|^2\le (N(t)V,V)\le c_2|V|^2
\]
holds for all $t\in[a,b]$ and $V\in\C^m$.
\end{lemma}
Note that the matrix $S$ is positive definite and its entries and eigenvalues are bounded with respect to $t\in[0,T]$, $\xi\in\R^n$ and $\eps\in(0,1]$. In addition, since by construction
\[
\lambda_{j+1,\eps}(t,\xi)-\lambda_{j,\eps}(t,\xi)\ge \omega(\eps)\lara{\xi},
\]
for all $t\in[0,T]$, $\xi\in\R^n$ and $\eps\in(0,1]$, then from \eqref{detS} it follows that 
the bound from below
\[
{\rm det}S(t,\xi)\ge \omega(\eps)^{m^2-m},
\]
holds. By applying Lemma \ref{lemmaJT} to $S$ we can therefore write
\[
\lara{SV,V}\ge c_1{\rm det}S(t,\xi)|V|^2\ge c_1\omega(\eps)^{m^2-m}|V|^2.
\]
More precisely, we can state the following lemma.
\begin{lemma}
\label{lemmaGR}
Let $S(t,\xi)$ be the symmetriser of the strictly hyperbolic matrix $A(t,\xi)\lara{\xi}^{-1}$ defined above. Then, there exist two positive constants $c_1$ and $c_2$ such that
\[
c_1\omega(\eps)^{m^2-m}|V|^2\le (S(t,\xi)V,V)\le c_2|V|^2
\]
holds for all $t\in[0,T]$, $\xi\in\R^n$, $\eps\in(0,1]$ and $V\in\C^m$.
\end{lemma}
\subsection{Energy estimates and proof of Theorem \ref{main_theo}}
We are now ready to introduce the energy
\[
E(t,\xi)=(S(t,\xi)V(t,\xi),V(t,\xi)).
\]
for the Cauchy problem \eqref{system_new},
\[
\begin{split}
D_t V&=A(t,\xi)V+B(t,\xi)V+\widehat{F}(t,\xi),\\
V|_{t=0}(\xi)&=V_0(\xi).
\end{split}.
\]
By straightforward computations we have
\[
\begin{split}
\partial_t E(t,\xi)&=(\partial_t SV,V)+ (S\partial_tV,V)+(SV,\partial_tV)\\
&=(\partial_t SV,V)+i((SA-A^\ast S)V,V)+i((SB-B^\ast S)V,V)-2\Im(S\widehat{F},V)\\
&=(\partial_t SV,V)+i((SB-B^\ast S)V,V)-2\Im(S\widehat{F},V)\\
&\le (\Vert\partial_t S\Vert+\Vert SB-B^\ast S\Vert +1)\vert V\vert^2+\Vert S\widehat{F}\Vert^2\\
&\le \max(\Vert\partial_t S\Vert+\Vert SB-B^\ast S\Vert +1,\Vert S\Vert^2)(\vert V\vert^2+\vert \widehat{F}\vert^2)\\
&\le c\,\omega(\eps)^{-N}(\vert V\vert^2+\vert \widehat{F}\vert^2)\\
\end{split}
\]
where, the last estimate follows from the regularisation of the equation coefficients, more precisely, from the fact that $\Vert\partial_t S\Vert\le c\,\omega(\eps)^{-1}$ and $\Vert SB-B^\ast S\Vert\le c\,\omega(\eps)^{-N}$ for some $c>0$ and $N\ge 1$, uniformly in $\eps\in(0,1]$.
Thus, an application of Lemma \ref{lemmaGR} yields 
\[
\partial_t E(t,\xi)\le cc_1^{-1}\omega(\eps)^{-N-m^2+m}E(t,\xi)+c\,\omega(\eps)^{-N}\vert \widehat{F}\vert^2.
\]

An application of Gronwall's lemma combined with the estimate from above in Lemma \ref{lemmaGR} gives the estimate
\[
|V|^2\le c_1^{-1}\omega(\eps)^{-m^2+m}E(t,\xi)\le c_1^{-1}\omega(\eps)^{-m^2+m}d_\eps \exp({cc^{-1}}\omega(\eps)^{-N-m^2+m}t),
\]
valid for all $\eps\in(0,1]$ and $t\in[0,T]$ with
\[
\begin{split}
d_\eps&=(E(0,\xi)V_0,V_0)+c\omega(\eps)^{-N}\sup_{0\le t\le T}\vert \widehat{F}(t,\xi)\vert^2T\\
&\le c'|V_0|^2+c\omega(\eps)^{-N}\sup_{0\le t\le T}\vert \widehat{F}(t,\xi)\vert^2T.
\end{split}
\]
Hence, by stating clearly the dependence in $\eps$, we can write
\[
|V_\eps^2|\le c_1^{-1}\omega(\eps)^{-m^2+m}(c'|V_{0,\eps}|^2+c\omega(\eps)^{-N}\sup_{0\le t\le T}\vert F_\eps(t,\xi)\vert^2T)\exp({cc^{-1}}\omega(\eps)^{-N-m^2+m}t).
\]
Since $V_{0,\eps}$ and $\widehat{F}_\eps$ are both Fourier transforms of $\gamma^s$-moderate nets for $s>1$, choosing $\omega(\eps)$ suitably, i.e., $cc^{-1}\omega(\eps)^{-N-m^2+m}=\ln(1/\eps)$ and using the Fourier characterisation of $\gamma^s$-moderate nets in Proposition \ref{prop_reg_gevrey_mod}, we can conclude that $U_\eps=\mathcal{F}^{-1}V_\eps$ is $\gamma^s$-moderate as well. Since the coefficients of the regularised equation we are studying are $C^\infty$-moderate in $t$ one can easily conclude that $U_\eps$ and consequently  the corresponding net $u_\eps$ solving the regularised Cauchy problem \eqref{CP_reg} are both $C^\infty([0,T];\gamma^s(\R^n))$-moderate. We have therefore proved that a very weak  solution of  order $s$ of the Cauchy problem \eqref{intro_CP} exists provided that coefficients and initial data are regularised in an appropriate way.
\begin{remark}
For simplicity we have here regularised the equation coefficients and the initial data using the same net $\omega(\eps)$ which has to be chosen of logarithmic type to ensure the existence of a very weak solution. A more careful analysis of the previous estimates shows that the choice of a logarithmic scale is essential when regularising the roots $\lambda_j$ and the lower order terms while the right hand-side $f$ and the initial data can be regularised setting $\omega(\eps)=\eps$.
\end{remark}

\section{Consistency with the classical results}
In this section we review some classical results concerning the hyperbolic Cauchy problem \eqref{intro_CP} valid when the coefficients of the principal part and therefore the corresponding roots are more regular, for instance of class $C^\alpha$ in $t$ with $\alpha\in(0,1]$. Colombini and Kinoshita in 2002 (see \cite{ColKi:02}) in one space dimension and later Garetto and Ruzhansky in \cite{GR:11} for any space dimension  proved that when the real roots $\lambda_j$ are of H\"older order $\alpha$ with respect to $t$, the lower order terms are continuous and of order $0\le l\le m-1$ and the right-hand side $f$ belongs to $C([0,T];\gamma^s_c(\R^n))$ then for any $g_k(x)\in \gamma^s_c(\Rn)$ ($k=1,\ldots,m$) the Cauchy problem \eqref{intro_eq}-\eqref{intro_ic} has a unique global solution
$u\in C^m([0,T]; \gamma^s(\Rn))$, provided that
\[
1< s<1+\min\left\{\alpha,\frac{m-l}{l}\right\}.
\]
Note that this well-posedness result has been obtained under the additional hypothesis that there exists $c>0$ such that
 \beq
\label{hyp_coincide}
|\lambda_i(t,\xi)-\lambda_j(t,\xi)|\le c|\lambda_k(t,\xi)-\lambda_{k-1}(t,\xi)|
\eeq
for all $1\le i,j,k\le m$, for all $t\in[0,T]$ and $\xi\in\R^n$, and it can be improved to 
\[
1< s<1+\frac{\alpha}{1-\alpha}
\]
when the roots are distinct. If the initial data $g_k$ belong to $\E'(\R^n)$ then the Cauchy problem \eqref{intro_CP} has a unique solution 
$u\in C^m([0,T]; \D'_{(s)}(\R^n))$, provided that
\[
1< s\le1+\min\left\{\alpha,\frac{m-l}{l}\right\},
\]
in the weakly hyperbolic case and that 
\[
1< s\le1+\frac{\alpha}{1-\alpha}
\]
in the strictly hyperbolic case.  

It is our aim now to show that when the roots are of class $C^\alpha$, $\alpha\in(0,1]$ in $t$ and the uniformity property \eqref{hyp_coincide} holds then any very weak solution of order $s$ obtained via Theorem \ref{main_theo} converges to the unique classical solution.  This requires the following preliminary result which compares the Fourier transform of a regularised compactly supported distribution with the Fourier transform of the distribution itself.
\begin{proposition}
\label{prop_FT_v}
Let $v\in\E'(\R^n)$. Then there exists $\nu>0$, $C>0$ and for all $q\in\mathbb{N}$ a constant $c_q>0$ such that 
\[
|\widehat{v\ast\rho_{\omega(\eps)}}(\xi)-\widehat{v}(\xi)|\le c_qC\omega(\eps)^q\esp^{\nu\lara{\xi}^{\frac{1}{s}}},
\]
for all $\xi\in\R^n$ and $\eps\in(0,1]$.
\end{proposition}

\begin{proof}
We recall that $\rho_{\omega(\eps)}$ has been defined in \eqref{def_rho}. By direct computations we have that
\[
|\widehat{v\ast\rho_{\omega(\eps)}}(\xi)-\widehat{v}(\xi)=\widehat{v}(\xi)(\widehat{\rho_{\omega(\eps)}}(\xi)-\widehat{\phi_{\omega(\eps)}}(\xi)+\widehat{\phi_{\omega(\eps)}}(\xi)-1),
\]
where $\phi\in\mathcal{S}^{(s)}(\R^n)$ with
\[
\int\phi(x)\, dx=1,\qquad \text{and}\qquad \int x^\alpha\phi(x)\, dx=0,\quad \text{for all $\alpha\neq 0$.}
\]
By the vanishing moments property of the mollifier $\phi$ it follows immediately that for all $q\in\mathbb{N}$ there exists $c_q>0$ such that
\beq
\label{est_1}
|\widehat{\phi_{\omega(\eps)}}(\xi)-1|=|\widehat{\phi}(\omega(\eps)\xi)-\widehat{\phi}(0)|\le c_q\omega(\eps)^q\lara{\xi}^q,
\eeq
holds for all $\xi\in\R^n$. We now write $\widehat{\rho_{\omega(\eps)}}(\xi)-\widehat{\phi_{\omega(\eps)}}(\xi)$ as
\[
\int_{\R^n}\esp^{-i\omega(\eps)y\xi}\phi(y)(\chi(y\omega(\eps)|\log\omega(\eps)|)-1)\, dy.
\]
Note that 
\[
|\omega(\eps)|\log\omega(\eps)|\le c\omega(\eps)^{\frac{1}{2}}.
\]
Since $\chi$ is compactly supported with $\chi(0)=1$, by Taylor's formula we have that
\begin{multline}
\label{est_2}
\biggl|\int_{\R^n}\esp^{-i\omega(\eps)y\xi}\phi(y)(\chi(y\omega(\eps)|\log\omega(\eps)|)-1)\, dy\biggr|\\
\le \sum_{|\alpha|=q}c_\alpha\int_{\R^n}|\phi(y)\partial^\alpha\chi(\theta y)y^\alpha|(\omega(\eps)|\log\omega(\eps)|)^q\, dy
\le c_q\omega(\eps)^{\frac{q}{2}}.
\end{multline}
Since $\E'(\R^n)\subseteq \E'_{(s)}(\R^n)$ and $|\widehat{v}(\xi)|\le c\esp^{\nu\lara{\xi}^{\frac{1}{s}}}$, by combining \eqref{est_1} with \eqref{est_2}  we conclude that
\[
|\widehat{v\ast\rho_{\omega(\eps)}}(\xi)-\widehat{v}(\xi)|\le c_qC\omega(\eps)^{\frac{q}{2}}\esp^{\nu\lara{\xi}^{\frac{1}{s}}}.
\]
This estimates proves the proposition by taking $2q$ rather than $q$.

\end{proof} 

We can now state and prove our consistency theorem.

\begin{theorem}
\label{theo_cons}
Let the Cauchy problem 
\beq
\label{CP_cons}
\begin{split}
D^m_t u-\sum_{j=1}^m\sum_{|\nu|= j} a_{\nu,j}(t)D_t^{m-j}D^\nu_x u-\sum_{j=1}^m\sum_{|\nu|<j}b_{\nu,j}(t)D^{m-j}_tD^\nu_x u &=f(t,x),\, t\in[0,T], x\in\R^n,\\
D_t^{k}u(0,x)&=g_k,\quad k=0,\cdots,m-1,
\end{split}
\eeq
have real roots $\lambda_j$ of H\"older order $\alpha$ in $t$, continuous lower order terms of order $0\le l\le m-1$, right-hand side $f\in C([0,T];\gamma^s_c(\R^n))$ and initial data $g_k(x)\in \gamma^s_c(\Rn)$ ($k=1,\ldots,m$) with
\[
1<s<1+\min\left\{\alpha,\frac{m-l}{l}\right\}.
\]
Assume in addition that the equation coefficients are compactly supported and that the uniform property \eqref{hyp_coincide} holds. 

Hence, any very weak solution $(u_\eps)_\eps$ of order $s$ converges in  $C([0,T], \gamma^s(\R^n))$  to the unique classical solution $u\in C^m([0,T], \gamma^s(\R^n))$ of the Cauchy problem \eqref{CP_cons}.

If the initial data $g_k$ belong to $\E'(\R^n)$ then any very weak solution $(u_\eps)_\eps$ of order $s$ with
\[
1<s\le1+\min\left\{\alpha,\frac{m-l}{l}\right\}.
\] 
converges in  $C([0,T], \D'_{(s)}(\R^n))$  to the unique classical solution $u\in C^m([0,T], \D'_{(s)}(\R^n))$ of the Cauchy problem \eqref{CP_cons}.
\end{theorem}
It is clear from the previous statement that the limit $u$ does not depend on the $C^\infty$-moderate regularisation of the equation roots and coefficients and on the $\gamma^s$-moderate regularisation of the initial data.
\begin{proof}

\emph{The case of Gevrey initial data}

We begin by observing that the right-hand side $f$ needs to be regularised only with respect to $t$ since it is already Gevrey in $x$ and that, for the same reason, the initial data $g_k$, $k=0,\dots,m-1$ do not need to be regularised. Let $(u_\eps)_\eps$ be a very weak solution of the Cauchy problem \eqref{CP_cons} that by Theorem \ref{main_theo} we know to exist. This means that there exists  $u_\eps\in C^m([0,T];\gamma^s(\R^n))$ which solves the regularised Cauchy problem  
\beq
\label{CP_c1}
\begin{split}
D^m_t u_\eps-\sum_{j=1}^m\sum_{|\nu|= j} a_{\nu,j,\eps}(t)D_t^{m-j}D^\nu_x u_\eps-\sum_{j=1}^m\sum_{|\nu|<j}b_{\nu,j,\eps}(t)D^{m-j}_tD^\nu_x u_\eps &=f_\eps(t,x),\\
D_t^{k}u_\eps(0,x)&=g_k, 
\end{split}
\eeq
where $t\in[0,T]$, $x\in\R^n$,  $k=0,\cdots,m-1$, the coefficients $a_{\nu,j,\eps}$ and $b_{\nu,j,\eps}$ are $C^\infty$-moderate regularisations of the coefficients $a_{\nu,j}$ and $b_{\nu,j}$, respectively and $f_\eps(t,x)=f(\cdot,x)\ast\varphi_{\omega(\eps)}$, for some suitable mollifier $\varphi_{\omega(\eps)}$ as in Section 2. Under the hypothesis of Theorem \ref{theo_cons} we know that there exists a $u\in C^m([0,T];\gamma^s(\R^n))$ such that
\beq
\begin{split}
\label{CP_c2}
 D^m_t u-\sum_{j=1}^m\sum_{|\nu|= j} a_{\nu,j}(t)D_t^{m-j}D^\nu_x u-\sum_{j=1}^m\sum_{|\nu|<j}b_{\nu,j}(t)D^{m-j}_tD^\nu_x u&=f(t,x),\\
D_t^{k}u(0,x)&=g_k, 
\end{split}
\eeq
for all $t\in[0,T]$, $x\in\R^n$,  $k=0,\cdots,m-1$. By comparing \eqref{CP_c1} with \eqref{CP_c2} we get that
\beq
\begin{split}
\label{CP_c3}
D^m_t\tilde{u}_\eps-\sum_{j=1}^m\sum_{|\nu|= j} a_{\nu,j,\eps}(t)D_t^{m-j}D^\nu_x\tilde{u}_\eps -\sum_{j=1}^m\sum_{|\nu|<j}b_{\nu,j,\eps}(t)D^{m-j}_tD^\nu_x \tilde{u}_\eps&=n_\eps(t,x),\\
D_t^{k}\tilde{u}_\eps(0,x)&=0,  
\end{split}
\eeq
where $\tilde{u}_\eps=u_\eps-u$ and
\begin{multline}
\label{def_n}
n_\eps(t,x)=f_\eps(t,x)-f(t,x) +
\sum_{j=1}^m\sum_{|\nu|= j} (a_{\nu,j,\eps}-a_{\nu,j})(t)D_t^{m-j}D^\nu_xu\\+\sum_{j=1}^m\sum_{|\nu|<j}(b_{\nu,j,\eps}-b_{\nu,j})(t)D^{m-j}_tD^\nu_x u.
\end{multline}
Note that since $\omega(\eps)\to 0$ as $\eps\to 0$ then $f_\eps\to f$ in $C([0,T];\gamma^s(\R^n))$ and, by regularisation of the coefficients $a_{\nu,j}$ and $b_{\nu,j}$, the nets $a_{\nu,j,\eps}-a_{\nu,j}$ and $b_{\nu,j,\eps}-b_{\nu,j}$ tend to $0$ uniformly on $[0,T]$. Hence, since $f\in C([0,T];\gamma^s_c(\R^n))$ and $u\in\gamma^s_c(\R^n)$ we can conclude that $n_\eps(t,x)$ is compactly supported with respect to $x$ and tends to $0$ in $C([0,T], \gamma^s(\R^n))$.

We now work on the Cauchy problem \eqref{CP_c3} as in the proof of Theorem 3 in \cite{GR:11}. In other words this means to reduce the equation in \eqref{CP_c3} to the first order system 
\[
\begin{split}
D_t V_\eps&=A_\eps(t,\xi)V_\eps+B_\eps(t,\xi)V_\eps+\widehat{F}_\eps(t,\xi),\\
V_\eps|_{t=0}(\xi)&=0,
\end{split}.
\]
to regularise once more the separated roots $\lambda_{j,\eps}$ as follows
\[
\wt{\lambda}_{j,\eps}(t)=(\lambda_{j,\eps}\ast\psi_\delta)(t)+j\delta^\alpha,\quad \delta\in(0,1],
\]
where $\psi\in\Cinfc(\R^n)$ with $\int\psi=1$ and to look for a solution of the type 
\beq
\label{VW}
V_\eps(t,\xi)=\esp^{-\mu(t)\lara{\xi}^{\frac{1}{s}}}(\det H_\eps)^{-1}H_\eps W_\eps,
\eeq
where $\mu\in C^1[0,T]$ will be determined in the sequel and
\[
H_\eps=\left(
    \begin{array}{ccccc}
      1 & 1 & 1 & \dots & 1\\
      \wt{\lambda}_{1,\eps}\lara{\xi}^{-1} &  \wt{\lambda}_{2,\eps}\lara{\xi}^{-1} &  \wt{\lambda}_{3,\eps}\lara{\xi}^{-1} & \dots &  \wt{\lambda}_{m,\eps}\lara{\xi}^{-1} \\
       \wt{\lambda}^2_{1,\eps}\lara{\xi}^{-2} &  \wt{\lambda}^2_{2,\eps}\lara{\xi}^{-2} &  \wt{\lambda}^2_{3,\eps}\lara{\xi}^{-2} & \dots & \wt{\lambda}^2_{m,\eps}\lara{\xi}^{-2} \\
      \dots & \dots & \dots & \dots & \dots\\
       \wt{\lambda}^{m-1}_{1,\eps}\lara{\xi}^{-m+1} & \wt{\lambda}^{m-1}_{2,\eps}\lara{\xi}^{-m+1} &  \wt{\lambda}^{m-1}_{3,\eps}\lara{\xi}^{-m+1} & \dots &  \wt{\lambda}^{m-1}_{m,\eps}\lara{\xi}^{-m+1}\\
    \end{array}
  \right).
\]
Note that, for the sake of the reader, the dependence in $\eps$ is expressed throughout the proof whereas the dependence on $\delta$ is hidden since at the certain point of the proof $\delta$ will be set to be a suitable power of $\lara{\xi}$. Note also that the extra regularisation of the roots $\lambda_j$ and the separation factor $j\delta^\alpha$ given above allow us to obtain uniform estimates in $\eps$ when deriving with respect to $t$ and estimating from below. More precisely, arguing as in Proposition 18 in \cite{GR:11} we have that
\[
|\partial_t\wt{\lambda}_{j,\eps}(t,\xi)|\le c\delta^{\alpha-1}\lara{\xi},
\]
\[
|\wt{\lambda}_{j,\eps}(t,\xi)-\wt{\lambda}_{i,\eps}(t,\xi)\ge \delta^\alpha\lara{\xi},\quad j>i, 
\]
and
\[
|\wt{\lambda}_{j,\eps}(t,\xi)-\lambda_{j,\eps}(t,\xi)|\le c\delta^\alpha\lara{\xi},
\]
uniformly with respect to $t\in[0,T]$, $\xi\in\R^n$, $\eps$ and $\delta$ in $(0,1]$.

Now, by following the proof of Theorem 3 in \cite{GR:11} we arrive at the energy estimate
\begin{multline*}
\partial_t |W_\eps(t,\xi)|^2=2{\rm Re} (\partial_t W_\eps(t,\xi),W_\eps(t,\xi))\\
=2\mu'(t)\lara{\xi}^{\frac{1}{s}}|W_\eps(t,\xi)|^2+2\frac{\partial_t\det H_\eps}{\det H_\eps}|W_\eps(t,\xi)|^2-2
{\rm Re}(H_\eps^{-1}\partial_t H_\eps W_\eps,W_\eps)\\
-2{\rm Im} (H_\eps^{-1}A_\eps H_\eps W_\eps,W_\eps)-2{\rm Im} (H_\eps^{-1}B_\eps H_\eps W_\eps,W_\eps) \\
-2{\rm Im} (\esp^{\mu(t)\lara{\xi}^{\frac{1}{s}}}(\det H_\eps) H_\eps^{-1}\widehat{F}_\eps,W_\eps)
\end{multline*}
and to the estimates
\begin{enumerate}
\item $|\frac{\partial_t\det H_\eps}{\det H_\eps}|\le c_1\delta^{-1}$,
\item $\Vert H_\eps^{-1}\partial_t H_\eps\Vert\le c_2\delta^{-1}$,
\item $\Vert H_\eps^{-1}A_\eps H_\eps-(H_\eps^{-1}A_\eps H_\eps)^\ast\Vert\le c_3\delta^\alpha\lara{\xi}$,
\item $\Vert H_\eps^{-1}B_\eps H_\eps-(H_\eps^{-1}B_\eps H_\eps)^\ast\Vert\le c_4\delta^{\alpha(1-m)}\lara{\xi}^{l-m+1}$.
\end{enumerate}
Note that by the uniform convergence in $t$ to $0$ of the nets $f_\eps-f$, $a_{\nu,j,\eps}-a_{\nu,j}$ and $b_{\nu,j,\eps}-b_{\nu,j}$ passing at the Fourier transform level we get that 
\[
|\widehat{F}_\eps(t,\xi)|\le c'\omega'(\eps)\esp^{-\nu\lara{\xi}^{\frac{1}{s}}},
\]
for some net $\omega'(\eps)$ tending to $0$ as $\eps\to 0$, for some constants $c',\nu>0$ and for all $\eps\in(0,1]$, $t\in[0,T]$ and $\xi\in\R^n$. Hence, making use of the four estimates above and arguing as in the proof of Theorem 3 in \cite{GR:11} we arrive at
\begin{multline*}
\partial_t |W_\eps(t,\xi)|^2   
\le (2\mu'(t)\lara{\xi}^{\frac{1}{s}}+C_1\delta^{-1}+C_2\delta^{\alpha}\lara{\xi}+C_3\delta^{\alpha(1-m)}\lara{\xi}^{l-m+1})
|W_\eps(t,\xi)|^2+\\
+C'\omega'(\eps)\esp^{(\mu(t)-\nu)\lara{\xi}^{\frac{1}{s}}}|W_\eps(t,\xi)|.
\end{multline*}
Set now $\delta=\lara{\xi}^{-\gamma}$ with $\gamma=\min\{\frac{1}{1+\alpha},\frac{m-l}{\alpha m}\}$ and $\mu(t)=(\mu(0)-\kappa t)$, with $\mu(0)$ and $\kappa>0$ to be determined later on. Assuming $|W_\eps(t,\xi)|\ge 1$ and recalling that by construction $-\gamma\alpha+1<\frac{1}{s}$, taking $\eps$ small enough  we have that
\begin{multline}
\label{energy1}
\partial_t |W_\eps(t,\xi)|^2\le \big(-2\kappa \lara{\xi}^{\frac{1}{s}}+C\lara{\xi}^{-\gamma\alpha+1}+
C'\omega'(\eps)\esp^{(\mu(0)-\nu)\lara{\xi}^{\frac{1}{s}}}\big)|W_\eps(t,\xi)|^2 \\
\le \big(-2\kappa \lara{\xi}^{\frac{1}{s}}+C\lara{\xi}^{-\gamma\alpha+1}+
C'\esp^{(\mu(0)-\nu)\lara{\xi}^{\frac{1}{s}}}\big)|W_\eps(t,\xi)|^2.\\
\end{multline}
At this point, setting $\mu(0)<\nu$ for $|\xi|$ large enough we conclude that $\partial_t |W_\eps(t,\xi)|^2\le 0$, i.e., 
\[
|W_\eps(t,\xi)|\le |W_{0,\eps}(t,\xi)|,
\]
for all $t\in(0,T]$ for $\eps$ small and $|\xi|$ large. Since, $V_{0,\eps}=0$ and the term $$\esp^{-\mu(t)\lara{\xi}^{\frac{1}{s}}}(\det H_\eps)^{-1}H_\eps$$ is invertible it turns out that $W_{0,\eps}(t,\xi)$ is identically $0$ as well and therefore $W_{\eps}(t,\xi)=0$ which is in contradiction with the assumption $|W_\eps(t,\xi)|\ge 1$. Hence, $|W_\eps(t,\xi)|$ must be $\le 1$ and 
\[
\partial_t |W_\eps(t,\xi)|^2\le \big(-2\kappa \lara{\xi}^{\frac{1}{s}}+C\lara{\xi}^{-\gamma\alpha+1})|W_\eps(t,\xi)|^2+
C'\omega'(\eps)\esp^{(\mu(0)-\nu)\lara{\xi}^{\frac{1}{s}}}.
\]
Assume $\mu(0)<\nu$ as above. For $|\xi|$ large (this is not restrictive) and for $-\gamma\alpha+1<\frac{1}{s}$ we get  
\[
|W_\eps(t,\xi)|^2\le \biggl(|W_\eps(0,\xi)|^2+C'T\omega'(\eps)\esp^{(\mu(0)-\nu)\lara{\xi}^{\frac{1}{s}}}\biggr)
= C'T\omega'(\eps)\esp^{(\mu(0)-\nu)\lara{\xi}^{\frac{1}{s}}}
\]
Since $\omega'(\eps)\to 0$ as $\eps>0$ it follows that $W_\eps(t,\xi)\to 0$ uniformly in $C([0,T];\gamma^s(\R^n))$. Passing now to $V_\eps$ defined in \eqref{VW} we obtain
\begin{multline*}
|V_\eps(t,\xi)|^2\le \biggl(\esp^{-(\mu(0)-\kappa t)\lara{\xi}^{\frac{1}{s}}}(|\det H_\eps)^{-1}H_\eps| \biggr)^2|W_\eps(t,\xi)|^2\\
\le \biggl(\esp^{-(\mu(0)-\kappa t)\lara{\xi}^{\frac{1}{s}}}(|\det H_\eps)^{-1}H_\eps| \biggr)^2\biggl(C'T\omega'(\eps)\esp^{(\mu(0)-\nu)\lara{\xi}^{\frac{1}{s}}}\biggr)\\
\le C'T\omega'(\eps)|(\det H_\eps)^{-1}H_\eps|^2\esp^{{(-\mu(0)-\nu+2\kappa T)}\lara{\xi}^{\frac{1}{s}}}.
\end{multline*}
Note that, by definition of $H_\eps$, the estimate
\beq
\label{form_H}
|(\det H_\eps)^{-1}H_\eps|\le c\lara{\xi}^{\gamma\alpha\frac{(m-1)m}{2}}
\eeq
holds uniformly in all the variables. Hence,
\[
|V_\eps(t,\xi)|^2\le C^{''}\omega'(\eps)\lara{\xi}^{\gamma\alpha(m-1)m}\esp^{{(-\mu(0)-\nu+2\kappa T)}\lara{\xi}^{\frac{1}{s}}}.
\]
Taking $2\kappa T<\mu(0)+\nu$ in the previous estimate we conclude by Fourier characterisation that $\wt{u}=u_\eps-u$ corresponding to $V_\eps$ converges to $0$ in $C([0,T];\gamma^s(\R^n))$ as $\eps\to 0$.

\emph{The case of distributional initial data}

If the initial data $g_k$ are compactly supported distributions for $k=1,\dots,m-1$, then the regularised nets $g_{k,\eps}=g_k\ast\rho_{\omega(\eps)}$, as defined in Subsection 2.2, are $\gamma^s$-moderate and $g_{k,\eps}\to g_k$ in $\E'(\R^n)$. We have therefore the regularised Cauchy problem
\beq
\label{CP_c4}
\begin{split}
D^m_t u_\eps-\sum_{j=1}^m\sum_{|\nu|= j} a_{\nu,j,\eps}(t)D_t^{m-j}D^\nu_x u_\eps-\sum_{j=1}^m\sum_{|\nu|<j}b_{\nu,j,\eps}(t)D^{m-j}_tD^\nu_x u_\eps &=f_\eps(t,x),\\
D_t^{k}u_\eps(0,x)&=g_{k,\eps}, 
\end{split}
\eeq
In addition, by the well-posedness results in \cite{GR:11} we know that the Cauchy problem \eqref{CP_c2} is well-posed in $C^m([0,T], \D'_{(s)}(\R^n))$. Hence, by comparing \eqref{CP_c4} with \eqref{CP_c2} we obtain
\[
\begin{split}
D^m_t\tilde{u}_\eps-\sum_{j=1}^m\sum_{|\nu|= j} a_{\nu,j,\eps}(t)D_t^{m-j}D^\nu_x\tilde{u}_\eps -\sum_{j=1}^m\sum_{|\nu|<j}b_{\nu,j,\eps}(t)D^{m-j}_tD^\nu_x \tilde{u}_\eps&=n_\eps(t,x),\\
D_t^{k}\tilde{u}_\eps(0,x)&=g_{k,\eps}-g_k,  
\end{split}
\]
with $\tilde{u}_\eps=u_\eps-u$ and $n_\eps$ as in \eqref{def_n}. Since the classical solution $u$ is a compactly supported ultradistribution with respect to $x$ we have that $n_\eps(t,x)$ tends to $0$ in $C([0,T], \D'_{(s)}(\R^n))$. Passing now to the corresponding first order system 
\[
\begin{split}
D_t V_\eps&=A_\eps(t,\xi)V_\eps+B_\eps(t,\xi)V_\eps+\widehat{F}_\eps(t,\xi),\\
V_\eps|_{t=0}(\xi)&=V_{0,\eps},
\end{split}.
\]
by the Fourier characterisation of nets of ultradistributions  and Proposition \ref{prop_FT_v} we easily see that 
\[
|\widehat{F}_\eps(t,\xi)|\le \omega'(\eps)\esp^{\nu\lara{\xi}^{\frac{1}{s}}},
\]
and
\[
|V_{0,\eps}(\xi)|\le \omega'(\eps)\esp^{\nu\lara{\xi}^{\frac{1}{s}}}
\]
for some net $\omega'(\eps)$ tending to $0$ as $\eps\to 0$, for some constant $\nu>0$ and for all $\eps\in(0,1]$, $t\in[0,T]$ and $\xi\in\R^n$.  Assume now that $|W_\eps(t,\xi)|\ge 1$. By arguing as in \eqref{energy1} with $-\gamma\alpha+1\le \frac{1}{s}$ by choosing $\kappa>0$ large enough we arrive at
\begin{multline*}
\partial_t |W_\eps(t,\xi)|^2\le \big(-2\kappa \lara{\xi}^{\frac{1}{s}}+C\lara{\xi}^{-\gamma\alpha+1}+
C'\omega'(\eps)\esp^{(\mu(t)+\nu)\lara{\xi}^{\frac{1}{s}}}\big)|W_\eps(t,\xi)|^2 \\
\le \big(-2\kappa \lara{\xi}^{\frac{1}{s}}+C\lara{\xi}^{-\gamma\alpha+1}+
C'\esp^{(\mu(0)-\kappa t+\nu)\lara{\xi}^{\frac{1}{s}}}\big)|W_\eps(t,\xi)|^2\le 0.\\
\end{multline*}
At the level of $V_\eps$ this means that 
\begin{multline*}
\label{last_estimate}
|V_\eps(t,\xi)|
=\esp^{-\mu(t)\lara{\xi}^{\frac{1}{s}}}\frac{1}{\det H_\eps(t,\xi)}|H_\eps(t,\xi)||W_\eps(t,\xi)|\le \\
\esp^{-\mu(t)\lara{\xi}^{\frac{1}{s}}}\frac{1}{\det H_\eps(t,\xi)}|H_\eps(t,\xi)||W_\eps(0,\xi)|=\\
\esp^{(-\mu(t)+\mu(0))\lara{\xi}^{\frac{1}{s}}}\frac{\det H_\eps(0,\xi)}{\det H_\eps(t,\xi)}|H_\eps(t,\xi)||H_\eps^{-1}(0,\xi)||V_\eps(0,\xi)|,
\end{multline*}
where, for $\gamma$ and $\delta$ as in above, we have
\beq
\label{form_VW}
\frac{\det H_\eps(0,\xi)}{\det H_\eps(t,\xi)}|H_\eps(t,\xi)||H_\eps^{-1}(0,\xi)|\le c\,\delta^{-\alpha\frac{(m-1)m}{2}}=c\lara{\xi}^{\gamma\alpha\frac{(m-1)m}{2}}.
\eeq
Hence,
\begin{multline}
\label{last_estimate1}
|V_\eps(t,\xi)|\le c\,\esp^{(-\mu(t)+\mu(0))\lara{\xi}^{\frac{1}{s}}}\lara{\xi}^{\gamma\alpha\frac{(m-1)m}{2}}|V_\eps(0,\xi)|\\
\le c_1\omega'(\eps)\,\esp^{(-\mu(t)+\mu(0))\lara{\xi}^{\frac{1}{s}}}\lara{\xi}^{\gamma\alpha\frac{(m-1)m}{2}}\esp^{\nu\lara{\xi}^{\frac{1}{s}}},
\end{multline}
for $|W_\eps(t,\xi)|\ge 1$.

When $|W_\eps(t,\xi)|\le 1$ by Gronwall's inequality we obtain the estimate
 \[
|W_\eps(t,\xi)|^2\le |W_\eps(0,\xi)|^2+C'T\omega'(\eps)\esp^{(\mu(0)+\nu)\lara{\xi}^{\frac{1}{s}}}.
 \]
 This implies
\begin{multline*}
|V_\eps(t,\xi)|^2
=\esp^{-2\mu(t)\lara{\xi}^{\frac{1}{s}}}\frac{1}{\det^2 H_\eps(t,\xi)}|H_\eps(t,\xi)|^2|W_\eps(t,\xi)|^2\le \\
\esp^{-2\mu(t)\lara{\xi}^{\frac{1}{s}}}\frac{1}{\det^2 H_\eps(t,\xi)}|H_\eps(t,\xi)|^2\biggl(|W_\eps(0,\xi)|^2+C'T\omega'(\eps)\esp^{(\mu(0)+\nu)\lara{\xi}^{\frac{1}{s}}}\biggr)=\\
\esp^{(-2\mu(t)+2\mu(0))\lara{\xi}^{\frac{1}{s}}}\frac{\det^2 H_\eps(0,\xi)}{\det^2 H_\eps(t,\xi)}|H_\eps(t,\xi)|^2|H_\eps^{-1}(0,\xi)|^2|V_\eps(0,\xi)|^2\\
+\esp^{-2\mu(t)\lara{\xi}^{\frac{1}{s}}}\frac{1}{\det^2 H_\eps(t,\xi)}|H_\eps(t,\xi)|^2C'T\omega'(\eps)\esp^{(\mu(0)+\nu)\lara{\xi}^{\frac{1}{s}}}.
\end{multline*}
By the properties of $V_{0,\eps}$ and the estimates \eqref{form_VW} and \eqref{form_H} we deduce that
\begin{multline}
\label{last_estimate2}
|V_\eps(t,\xi)|^2\le c_2\omega'(\eps)^2\esp^{(2\kappa t+2\nu)\lara{\xi}^{\frac{1}{s}}}\lara{\xi}^{\gamma\alpha(m-1)m}\\
+c_2\omega'(\eps)\esp^{(-\mu(0)+2\kappa T+\nu)\lara{\xi}^{\frac{1}{s}}}\lara{\xi}^{\gamma\alpha(m-1)m},
\end{multline}
for $|W_\eps(t,\xi)|\le 1$. Concluding, by combining \eqref{last_estimate1} with \eqref{last_estimate2} we have that $V_\eps(t,\xi)$ tends to $0$ as a net of ultradistributions uniformly with respect to $t$ and therefore $u_\eps$ tends to $u$ in $C([0,T],\D'_{(s)}(\R^n))$.
\end{proof}
 
\begin{remark}
The previous proof clearly shows that Theorem \ref{theo_cons} can be stated for first order $m\times m$ hyperbolic systems in Sylvester (or more in general block Sylvester) form. In other words, the existence of very weak solutions for systems of this type is consistent with the classical Gevrey or ultradistributional results.
\end{remark} 

\section{Very weak solvability of first order hyperbolic systems}

We end this paper by considering the following Cauchy problem
\beq
\label{CP_syst_fin}
\begin{split}
D_t u-A(t,D_x)u-B(t)u&=F(t,x),\quad x\in\R^n,\, t\in[0,T],\\
u|_{t=0}&=g_0,
\end{split}
\eeq
where $A$ and $B$ are $m\times m$ matrices of first order and zero order differential operators, respectively, with $t$-dependent coefficients, $F$, $u$ and $u_0$ are column vectors with $m$ entries. We work under the following set of hypotheses: 

\begin{itemize}
\item[(h1)] the system coefficients are compactly supported in $t$ and the eigenvalues of the matrix $A$ are real and bounded with respect to $t\in[0,T]$,
\item[(h2)] the entries of the matrix $B$ belong to $\E'([0,T])$, 
\item[(h3)] $F\in( \E'([0,T])\otimes\E'(\R))^m$ and $g_0\in \E'(\R^n)^m$.
\end{itemize}

We want to investigate the very weak solvability of the Cauchy problem \eqref{CP_syst_fin}, in the sense that we want to understand if, given $s>1$, there exist

\begin{itemize}
\item[(i)] $C^\infty$-moderate regularisations $A_\eps$ and  $B_\eps$ of the matrices $A$ and $B$, respectively,
\item[(ii)] a $C^\infty([0,T];\gamma^s(\R^n))$-moderate regularisation $F_\eps$ of the right-hand side $F$, and
\item[(iii)] a $\gamma^s$-moderate regularisations $g_{0, \eps}$ of the initial data $g_0$, 
\end{itemize}
such that $(u_\eps)_\eps$ solves the regularised problem
\[
\begin{split}
D_t u_\eps-A_\eps(t,D_x)u_\eps-B_\eps(t)u_\eps&=F_\eps(t,x),\\
u_\eps|_{t=0}&=g_{0,\eps},
\end{split}
\]
for $t\in[0,T], x\in\R^n$ and $\eps\in(0,1]$, and is $C^\infty([0,T];\gamma^s(\R^n))$-moderate. Note that since we are dealing with matrices here the moderateness is meant in every entry. 

As observed already in the introduction and in the previous sections, the proof of Theorem \ref{main_theo} for $m$ order scalar equations is equivalent to the proof of the very weak solvability of a first order system of pseudodifferential equations with principal part in Sylvester form. This result of very weak solvability for systems can be easily adapted to a slightly more general principal part, i.e., to a principal part in block Sylvester form. From this fact it follows that the proof method of Theorem \ref{main_theo} directly implies the very weak solvability of the Cauchy problem \eqref{CP_syst_fin} if the hyperbolic system in \eqref{CP_syst_fin} can be transformed into a hyperbolic system with principal part in block Sylvester form. In the sequel we show that not only this Sylvester transformation is possible but also that the eigenvalues are preserved. This is due to the reduction to block Sylvester form given by d'Ancona and Spagnolo in \cite{DS} which, for the sake of the reader, we recall in the following subsection. 

\subsection{Reduction to block Sylvester form}

We begin by considering the cofactor matrix $L(t,\tau,\xi)$ of $\tau I-A(t,\xi)$ where $I$ is the $m\times m$ identity matrix. By applying the corresponding operator $L(t,D_t,D_x)$ to \eqref{CP_syst_fin} we transform the system
\[
D_t u-A(t,D_x)u-B(t)u=F(t,x)
\]
into
\beq
\label{red_1}
\delta(t,D_t,D_x)Iu-C(t,D_x,D_t)u=G(t,x),
\eeq
where $\delta(t,\tau,\xi)={\rm det}(\tau I-A(t,\xi))$, $C(t,D_x,D_t)$ is the matrix of  lower order terms (differential operators of order $m-1$) and $G(t,x)=L(t,D_t,D_x)F(t,x)$. Note that $\delta(t,D_t,D_x)$ is the operator
\[
D_t^m+\sum_{h=0}^{m-1}b_{m-h}(t,D_x)D_t^h,
\]
with $b_{m-h}(t,\xi)$ homogeneous polynomial of order $m-h$.

We now transform this set of scalar equations of order $m$ into a first order system of size $m^2\times m^2$ of pseudodifferential equations, by setting
\[
U=\{D_j^{j-1}\lara{D_x}^{m-j}u\}_{j=1,2,\dots,m},
\]
where $\lara{D_x}$ is the pseudodifferential operator with symbol $\lara{\xi}$. More precisely, we can write \eqref{red_1} in the form
\[
D_tU-\mathcal{A}(t,D_x)U+\mathcal{L}(t,D_x)U=\mathcal{R}(t,x),
\]
where $\mathcal{A}$ is a $m^2\times m^2$ matrix made of $m$ identical blocks of the type
\begin{multline*}
\lara{D_x}\cdot\\
\left(\begin{array}{cccccc}
                              0 & 1 & 0 & \cdots & 0 \\
                               0 & 0 & 1 &  \cdots & 0\\
                               \vdots & \vdots & \vdots & \vdots & \vdots\\
                              -b_m(t,D_x)\lara{D_x}^m& -b_{m-1}(t,D_x)\lara{D_x}^{-m+1} & \cdots  & \cdots & -b_1(t,D_x)\lara{D_x}^{-1}\\
                              \end{array}
                           \right),
                           \end{multline*}
the matrix $\mathcal{L}$  of the lower order terms is made of $m$ blocks of  size $m\times m^2$ of the type
\[
\left(\begin{array}{cccccc}
                              0 & 0 & 0 & \cdots & 0 & 0 \\
                               0 & 0 & 0 &  \cdots &  0 & 0\\
                               \vdots & \vdots & \vdots & \vdots & \vdots\\
                              l_{j,1}(t, D_x)& l_{j,2}(t,D_x) & \cdots  & \cdots & l_{j,m^2-1}(t,D_x) & l_{j,m^2}(t,D_x)
                                                     \end{array}
                           \right), 
\]
with $j=1,\dots,m$, and finally the right-hand side $\mathcal{R}$ is a $m^2\times 1$ matrix with $m$ column blocks of size $m\times 1$ of the type
\[
\left(\begin{array}{c}
                              0 \\
                               0 \\
                               \vdots \\
                              r_j(t,x),
                                                     \end{array}
                           \right)
\]
$j=1,\dots,m$.
By construction the matrices $\mathcal{A}$ and $\mathcal{B}$ are made by pseudodifferential operators of order $1$ and $0$, respectively. 

Concluding, the Cauchy problem \eqref{CP_syst_fin} has been transformed into
\beq
\label{CP_syst_Syl}
\begin{split}
D_tU-\mathcal{A}(t,D_x)U+\mathcal{L}(t,D_x)U&=\mathcal{R}(t,x),\\
U_{t=0} &= \{D_j^{j-1}\lara{D_x}^{m-j}g_0\}_{j=1,2,\dots,m}.
\end{split}
\eeq
This is a Cauchy problem of first order pseudodifferential equations with principal part in block Sylvester form. The size of the system is increased from $m\times m$ to $m^2\times m^2$ but the system is still hyperbolic, since the eigenvalues of any block of $\mathcal{A}(t,\xi)$ are the eigenvalues of the matrix $\lara{\xi}^{-1}A(t,\xi)$.

\subsection{Proof of Theorem \ref{main_theo_syst} }
The first step in the proof of Theorem \ref{main_theo_syst} is the reduction of \eqref{CP_syst_fin} into \eqref{CP_syst_Syl}. Since we work under the set of hypotheses (h1), (h2) and (h3) we can regularise the system eigenvalues and coefficients and the initial data as in the scalar equation case. Making use of the block Sylvester form of the matrix $\mathcal{A}$ and of the corresponding regularised version we can construct a symmetriser which is block diagonal (made of $m$ blocks of size $m\times m$) and repeat the arguments of the proof of Theorem \ref{main_theo}. This easily leads to the existence of a very weak solution for any $s>1$.


\begin{thebibliography}{CDGS79}

\bibitem[BB09]{BenBou:09}
K.~Benmeriem and C.~Bouzar.
\newblock Generalized {G}evrey ultradistributions.
\newblock {\em New York J. Math.}, 15:37--72, 2009.



\bibitem[Bro80]{Bronshtein:TMMO-1980}
M.~D. Bron{\v{s}}te{\u\i}n.
\newblock The {C}auchy problem for hyperbolic operators with characteristics of
  variable multiplicity.
\newblock {\em Trudy Moskov. Mat. Obshch.}, 41:83--99, 1980.


\bibitem[CDGS79]{Colombini-deGiordi-Spagnolo-Pisa-1979}
F.~Colombini, E.~De~Giorgi, and S.~Spagnolo.
\newblock Sur les {\'e}quations hyperboliques avec des coefficients qui ne
  d{\'e}pendent que du temps.
\newblock {\em Ann. Scuola Norm. Sup. Pisa Cl. Sci. (4)}, 6(3):511--559, 1979.



\bibitem[CJS83]{CJS:Pisa-1983}
F.~Colombini, E.~Jannelli, and S.~Spagnolo.
\newblock Well-posedness in the {G}evrey classes of the {C}auchy problem for a
  nonstrictly hyperbolic equation with coefficients depending on time.
\newblock {\em Ann. Scuola Norm. Sup. Pisa Cl. Sci. (4)}, 10(2):291--312, 1983.

\bibitem[CJS87]{Colombini-Jannelli-Spagnolo:Annals-low-reg}
F.~Colombini, E.~Jannelli, and S.~Spagnolo.
\newblock Nonuniqueness in hyperbolic {C}auchy problems.
\newblock {\em Ann. of Math. (2)}, 126(3):495--524, 1987.

\bibitem[CK02]{ColKi:02}
F.~Colombini and T.~Kinoshita.
\newblock On the {G}evrey well posedness of the {C}auchy problem for weakly
  hyperbolic equations of higher order.
\newblock {\em J. Differential Equations}, 186(2):394--419, 2002.


\bibitem[CS82]{Colombini-Spagnolo:Acta-ex-weakly-hyp}
F.~Colombini and S.~Spagnolo.
\newblock An example of a weakly hyperbolic {C}auchy problem not well posed in
  {$C^{\infty }$}.
\newblock {\em Acta Math.}, 148:243--253, 1982.

\bibitem[dAKS04]{dAKS:04}
{P. d'Ancona, T. Kinoshita and S. Spagnolo}.
\newblock{Weakly hyperbolic systems with H\"older continuous coefficients.}
\newblock{\em J. Differential Equations}, {203(1)}, 64Ð81, 2004.

\bibitem[dAKS08]{dAKS:08}
{P. d'Ancona, T. Kinoshita and S. Spagnolo}.
\newblock On the 2 by 2 weakly hyperbolic systems.
\newblock{\em Osaka J. Math.}, {45(4)}, 921Ð939, 2008.

\bibitem[DS98]{DS}
P.~D'Ancona and S.~Spagnolo.
\newblock Quasi-symmetrization of hyperbolic systems and propagation of the
  analytic regularity.
\newblock {\em Boll. Unione Mat. Ital. Sez. B Artic. Ric. Mat. (8)},
  1(1):169--185, 1998.


\bibitem[GR12]{GR:11}
C.~Garetto and M.~Ruzhansky.
\newblock On the well-posedness of weakly hyperbolic equations with
  time-dependent coefficients.
\newblock {\em J. Differential Equations}, 253(5):1317--1340, 2012.

\bibitem[GR13]{GR:12}
C.~Garetto and M.~Ruzhansky.
\newblock Weakly hyperbolic equations with non-analytic coefficients and lower
  order terms.
\newblock {\em Math. Ann.}, 357(2):401--440, 2013.

\bibitem[GR14]{GR:14}
C.~Garetto and M.~Ruzhansky.
\newblock Hyperbolic  second order equations with non-regular time dependent coefficients
\newblock To appear in  {\em Arch. Rat. Mech. Appl.}, 2014.

\bibitem[HdH01]{Hormann-de-Hoop:AAM-2001}
G.~H{\"o}rmann and M.~V. de~Hoop.
\newblock Microlocal analysis and global solutions of some hyperbolic equations
  with discontinuous coefficients.
\newblock {\em Acta Appl. Math.}, 67(2):173--224, 2001.



\bibitem[J09]{J:09}
E. Jannelli,
\newblock The hyperbolic symmetrizer: theory and applications.
\newblock{\em in Advances in Phase Space Analysis of PDEs}, Birkh\"auser, 78, 113--139, 2009.

\bibitem[JT11]{JT}
E. Jannelli and G. Taglialatela,
\newblock Homogeneous weakly hyperbolic equations with time dependent analytic coefficients.
\newblock{\em J. Differential Equations}, 251 (2011), 995--1029.


\bibitem[KY06]{KY:06}
{ K. Kajitani and Y. Yuzawa.}
\newblock  The Cauchy problem for hyperbolic systems with H\"older continuous coefficients with respect to the time variable.
{\em Ann. Sc. Norm. Super. Pisa Cl. Sci.}, {5(4)}, 465--482, 2006.



\bibitem[MB99]{Marsan-Bean}
D.~Marsan and C.~J. Bean.
\newblock Multiscaling nature of sonic velocities and lithology in the upper
  crystalline crust: {E}vidence from the {KTB} main borehole.
\newblock {\em Geophys. Res. Lett.}, 26:275--278, 1999.

\bibitem[Nis83]{Nishitani:BSM-1983}
T.~Nishitani.
\newblock Sur les {\'e}quations hyperboliques {\`a} coefficients
  h{\"o}ld{\'e}riens en {$t$} et de classe de {G}evrey en {$x$}.
\newblock {\em Bull. Sci. Math. (2)}, 107(2):113--138, 1983.

\bibitem[NR10]{NiRo:10}
F.~Nicola and L.~Rodino.
\newblock {\em Global pseudo-differential calculus on {E}uclidean spaces},
  volume~4 of {\em Pseudo-Differential Operators. Theory and Applications}.
\newblock Birkh{\"a}user Verlag, Basel, 2010.



\bibitem[Rod93]{Rod:93}
L.~Rodino,
\newblock {\em Linear partial differential operators in Gevrey spaces}.
\newblock World Scient. Publish., River Edge, NJ, 1993.



\bibitem[Teo06]{Teo:06}
N.~Teofanov.
\newblock Modulation spaces, {G}elfand-{S}hilov spaces and pseudodifferential
  operators.
\newblock {\em Sampl. Theory Signal Image Process.}, 5(2):225--242, 2006.

\bibitem[Yu05]{Yu:05}
{Y. Yuzawa}.
\newblock The Cauchy problem for hyperbolic systems with H\"older continuous coefficients with respect to time.
{\em J. Differential Equations}, {219(2)}, 363--374, 2005.

\end{thebibliography}
\end{document}